\documentclass[12pt,a4paper]{article}%
\usepackage{amsmath}
\usepackage{graphicx}
\usepackage{epsfig}%
\usepackage{amsfonts}%
\usepackage{amssymb}
\usepackage{algorithm}
\usepackage{algorithmic}
\usepackage{color}

\newtheorem{theorem}{Theorem}[section]

\newenvironment{proof}[1][Proof]{\noindent \emph{#1.} }{\hfill \
\rule{0.5em}{0.5em}}

\newcommand{\herm}{^{\text{\normalfont\scriptsize\textsf H}}}
\newcommand{\conj}[1]{\overline{#1}}

\makeatletter\@addtoreset{equation}{section}\makeatother
\makeatletter\@addtoreset{figure}{section}\makeatother
\makeatletter\@addtoreset{table}{section}\makeatother
\begin{document}

\title{Computing the density of states for optical spectra 
by low-rank and QTT tensor approximation}

\author{Peter Benner\thanks{Max Planck Institute for Dynamics of Complex 
Systems, Sandtorstr.~1, D-39106 Magdeburg, Germany
({\tt benner@mpi-magdeburg.mpg.de})}\and
        Venera Khoromskaia\thanks{Max Planck Institute for
        Mathematics in the Sciences, Leipzig;
        Max Planck Institute for Dynamics of Complex Systems, Magdeburg ({\tt vekh@mis.mpg.de}).}
        \and Boris N. Khoromskij\thanks{Max Planck Institute for
        Mathematics in the Sciences, Inselstr.~22-26, D-04103 Leipzig,
        Germany ({\tt bokh@mis.mpg.de}).}
        \and
        Chao Yang\thanks{Berkeley Labs, Berkeley, USA ({\tt cyang@lbl.gov}).}
        \\        
        }

  \date{}

\maketitle
\begin{abstract}
  In this paper, we introduce a new interpolation scheme to approximate the density 
  of states (DOS) for a class of rank-structured matrices with application to 
  the Tamm-Dancoff approximation (TDA) of the Bethe-Salpeter equation (BSE).
 The presented approach for approximating the DOS is based on two main techniques.
First, we propose an economical method for calculating the traces of parametric 
matrix resolvents at interpolation points by taking advantage of
the block-diagonal plus low-rank matrix structure described in 
\cite{BeKhKh_BSE:15,BDKK_BSE:16} for the BSE/TDA problem.  
Second, we show that a regularized or smoothed DOS discretized on a fine grid of size $N$
can be accurately represented by a low rank quantized tensor train (QTT) tensor 
{  that can be determined through a least squares fitting procedure}.
The latter provides good approximation properties 
for strictly oscillating DOS functions with multiple gaps, and requires asymptotically much
fewer ($O(\log N)$)  functional calls   compared with the full grid size $N$.
This approach allows us to overcome the computational difficulties  
of the traditional schemes by avoiding both the need of stochastic sampling
and interpolation by problem independent functions like polynomials etc. 
Numerical tests indicate that the QTT approach  yields   accurate recovery of DOS 
associated with problems that contain  relatively large spectral gaps. 
The QTT tensor rank only weakly depends on the size of a molecular system
 which paves the way for treating large-scale spectral problems.  
\end{abstract}



\section{Introduction}\label{Introd:MP2}

Numerical approximation of the density of states (DOS)  
or spectral density   (see \S\ref{ssec:DOS_def})   of large matrices
is one of the challenging problems arising in the prediction of electronic, vibrational and  
thermal properties of molecules and crystals and many other applications. %
This topic, first developed in condensed matter physics 
\cite{DuCy:70,WhBlu:72,Turek:88,DrSa:93,Wang:94},  has long since attracted    
interest in the community of numerical linear algebra \cite{DorHo:00,GolVan:13,TrefEmbr:05},
see also a survey on commonly used methodology for approximation of DOS 
for large matrices of  general   structure \cite{LinSaadYa:15}. 
Most traditional methods are based on a polynomial or fractional-polynomial 
interpolation of the DOS regularized by Gaussians or Lorentzians, and  computing  
traces of certain matrix valued functions, say matrix resolvents or polynomials,
defined  at a large set of interpolation points within the spectral interval of interest.
Furthermore, the trace calculations  are typically accomplished with stochastic sampling over 
a large number of random vectors   \cite{LinSaadYa:15}.  

Since the size of matrices resulting from real life applications is
usually large (in quantum mechanics it scales as a polynomial of the molecular size), 
and the DOS of these matrices often exhibits very complicated shape, the above mentioned methods become
prohibitively expensive. 
Moreover, the algorithms based on polynomial interpolants have poor approximating
properties when the spectrum of a matrix exhibits gaps or highly oscillating non-regular shapes, 
as is the case in electronic  structure calculations.
Furthermore, stochastic sampling leads to poor Monte Carlo estimates with slow 
convergence rates and low accuracy.

In this paper we present a new method to  efficiently and accurately approximate the 
DOS for large rank-structured symmetric matrices. The approach amounts to   estimating   the DOS by 
evaluating matrix functions of structured matrices, in particular traces of the matrix resolvent. 
Our main contribution is to   perform each function evaluation at low cost and to reduce 
the total number of function evaluations in the case of fine  representation grid.

 We apply this approximation to the Bethe-Salpeter equation (BSE), which is a widely  
used  model for {\it ab initio} estimation of the absorption spectra for molecules or surfaces 
of solids \cite{BeSa:51,Hedin,StScuFr:98,ReOlRuOni:02,OniReRu:02,ReTouSa1:13}.
In particular, we use the  recently developed   low-rank structured representation of the BSE 
Hamiltonian, which was introduced and analyzed in \cite{BeKhKh_BSE:15}.
An efficient and structured eigenvalue solver for this block-diagonal plus low-rank 
representation of the BSE Hamiltonian as well as to its symmetric positive definite 
surrogate obtained by the Tamm-Dancoff approximation (TDA) is described in \cite{BDKK_BSE:16}.

The approach we take here to approximate the DOS of the BSE Hamiltonian relies on the 
Lorentzian blurring \cite{HayHeKe:72}. The most computationally  expensive  
part of the calculation is reduced to the evaluation of traces of shifted matrix inverses.
Our method is based on two main ingredients. 
First, we propose an economical method for calculating traces of parametric 
matrix resolvents at interpolation points by  taking advantage of  
the block-diagonal plus low-rank BSE/TDA matrix structure described in 
\cite{BeKhKh_BSE:15,BDKK_BSE:16},  which enables the direct inversion of the shifted
Hamiltonian within the same matrix structure.
This  allows us to overcome the computational difficulties  
of the traditional schemes and avoid the need of stochastic sampling.
 Second,  we show that a regularized or smoothed DOS can be accurately approximated by 
a low rank   QTT tensor  \cite{KhQuant:09} that can be determined through 
a least squares procedure.  
The accuracy of approximation and interpolation is controlled by $\epsilon$-truncation
of the corresponding matrix/tensor ranks.

Our fast method for calculating traces of matrix resolvents
for the family of rank-structured matrices exhibits almost linear 
asymptotic complexity scaling with respect to the matrix size. 
We introduce  an explicit rank-structured representation of the matrix inverse 
which can be evaluated efficiently by using the Sherman-Morrison-Woodbury formula.
Note that the diagonal plus low-rank approximation to the BSE  
Hamiltonian introduced in \cite{BeKhKh_BSE:15} employs the low-rank
approximation to the two-electron integrals tensor in the form of a Cholesky factorization 
\cite{VeKhBoKhSchn:12}.   
An efficient structured solver designed to calculate a number of minimal eigenvalues of 
the block-diagonal plus low-rank representation of the BSE/TDA matrices 
is described in \cite{BDKK_BSE:16}.

Another novelty of this paper is the application of the QTT tensor approximation 
to the DOS sampled on a fine grid, which results in a long vector. 
The QTT approximation method was introduced and analyzed for function related 
vectors in \cite{KhQuant:09}. 
It was proven that for a length-$N$ vector obtained from the discretization of 
a classical function (complex exponentials, polynomials, Gaussians etc.),
its QTT image constructed in the $L$-dimensional tensor space with $L=\log_2 N$ exhibits 
an amazingly low separation rank $r_{qtt}$. This rank parameter $r_{qtt}$ appears to be 
independent of the size of the original vector.
Thus the use of QTT tensor compression reduces the number of  representation parameters
from $N$ to $2 r_{qtt}^2 \log_2 N $, 
which allows asymptotically a much smaller number of functional calls, 
$O(\log N)$, to reconstruct the DOS function in the QTT parametrization. 
This might be beneficial in the limit of a large number of representation points $N$
since each functional evaluation of the DOS is highly expensive requiring 
computation of some matrix valued functions.

For example, for a vector of size $N=2^L$ representing the exponential function,  
its reshape (folding) into an $L$-dimensional tensor of size 
${\underbrace{2\times \cdots \times 2}}_{L-fold}$ with modes equal to $2$,
yields a QTT tensor of rank $r_{qtt}=1$, which means the reduction of storage
from $2^L$ to $ 2 \log_2 N=2 L$. For a complex exponential vector there holds $r_{qtt}=2$, then storage 
reduces from $N$ to $8 \log_2 N $.
Similar low rank QTT representations were proven for a wide class of functions~\cite{khor-survey-2014}, 
including strongly oscillating functions of nontrivial shape, see for example 
\cite{VeBoKh:Ewald:14,KhVe:16}  and the new results in \S\ref{ssec: QTT_ranks_DoS} below. 
For a general class of functional vectors, one computes an $\varepsilon$-rank QTT approximation
which leads to a storage size with logarithmic scaling in $N$.  

Numerical tests for moderate size molecules confirm the closeness of
DOS for the TDA model to those computed on the exact BSE spectrum. 
We also justify that the simplified block-diagonal plus low-rank approximation recovers well
the main  landscape and shape details of the DOS curve on the whole energy interval  and
check the precision of the low-rank QTT approximation to  the length-$N$ vector representing the DOS.
We demonstrate the almost linear complexity scaling of the trace calculation algorithm 
applied to TDA matrices of different size.
 We then show by numerical tests that the low-rank QTT tensor interpolation 
 scheme requires only a small number  
 of adaptively chosen samples  in the $N$-vector discretizing the DOS. 
   For instance, a polynomial interpolant of degree $p$ needs  $p+1$ interpolation 
points (functional calls) for the representation of a function on a large $N$-grid.
However, in the case of highly oscillating DOS functions of interest one should impose $p=O(N)$.
On the contrary, the QTT interpolant over $O(\log N)$ interpolation points provides a   
rather accurate representation of the functional $N$-vector of the DOS.  
 
We also discuss the opportunity to reduce the cost of multiple trace calculations 
for the parametric matrix resolvent and, finally, describe modifications necessary
to calculate the optical absorption spectrum via a rank-structured BSE model.

The rest of the paper is structured as follows. In Section \ref{sec:pre}, we recall the main
prerequisites for the description of our method including the rank-structured 
approximation to the BSE/TDA matrix, basic notions of  the  regularization of DOS by Lorentzians
and a short summary on the existing methods for matrices of general structure. 
Section \ref{sec:Rstruc_DOS} discusses the main
techniques of the presented method and the corresponding analysis in Theorems \ref{thm:Trace_cost} and 3.2.
The numerical tests confirm the linear scaling of our algorithm in the size of the grid
on which the DOS is evaluated.
Section \ref{ssec:QTT_DOS} presents a summary of the QTT tensor approximation of function related vectors 
and the analysis of the QTT tensor ranks of the DOS, see Theorem \ref{thm:QTT_R_Gaus_Broad}.  
In Section \ref{ssec: QTT_cross_DoS} the ACA based QTT interpolation is applied to the discretized 
DOS, where the quality of the interpolation is illustrated numerically.
The beneficial features of the new computational schemes are verified by extensive 
numerical experiments on the examples of various molecular systems.
Section  \ref{sec:BSE_case} outlines the extension of the approach 
to the case of full BSE system.
Conclusions  summarize  the main results and address  the application perspectives.

\section{Main prerequisites and outline of initial applications} \label{sec:pre}

 \subsection{Rank-structured approximation to BSE matrix}\label{ssec:BSE_setting}

In this paper we describe a method for efficient  and accurate  approximation of the
DOS for large rank-structured symmetric matrices. Our basic application is concerned 
with estimating the DOS and the absorption spectrum for the Bethe-Salpeter problem  
describing the excitation energies of molecules.    
 
The $2\times 2$-block matrix representation of the Bethe-Salpeter Hamiltonian (BSH) 
leads to the following eigenvalue problem. 
\begin{equation} \label{eqn:BSE-GW1}
   H
 \begin{pmatrix}
 {\bf x}_k\\
{\bf y}_k\\
\end{pmatrix}
\equiv
\begin{pmatrix}
{ A}   &  { B} \\
{- B}^\ast  &  {- A}^\ast  \\
\end{pmatrix}
\begin{pmatrix}
 {\bf x}_k\\
{\bf y}_k\\
\end{pmatrix}
= \omega_k
\begin{pmatrix}
 {\bf x}_k\\ {\bf y}_k\\
\end{pmatrix},
\end{equation}
where the matrix blocks of size $n \times n$, with $n=N_{ov}=N_{o}(N_{b}-N_{o})$,
are defined by
\begin{equation}
 A=  \boldsymbol{\Delta \varepsilon} + V - \widehat{W},\quad
 B= {V} - \widetilde{W},
\label{eq:AB_ex}
\end{equation}
and eigenvalues $\omega_k$ correspond to the excitation energies. 
Here $\boldsymbol{\Delta \varepsilon}$ is a diagonal matrix and
$$
{ V}=[v_{ia,jb}]\quad a, b \in {\cal I}_{v}:=\{N_{o}+1,\ldots,N_{b}\},
\quad i,j\in {\cal I}_{o}:=\{1,\ldots,N_{o}\},
$$
is the rank-$R_B$ two-electron integrals (TEI) matrix projected onto the Hartree-Fock 
molecular orbital basis, where $N_b$ is the number of Gaussian type orbital (GTO)  
basis functions and $N_{o}$ denotes the number of occupied orbitals \cite{BeKhKh_BSE:15}.

 The method for solving the Bethe-Salpeter equation (BSE) using 
low-rank factorizations of the generating matrices has been introduced in \cite{BeKhKh_BSE:15}.
It is based on  a  tensor-structured grid-based Hartree-Fock (HF) solver which
provides not only the full set of eigenvalues and HF orbitals,
but also the two-electron integrals tensor in the form of a low-rank Cholesky factorization,
see \cite{VeKhorTromsoe:15} and references therein.

 The matrix $V$ inherits its low rank from the two-electron integrals tensor,
 and $\widetilde{W}$ is also proven to have a small $\epsilon$-rank (see \cite{BeKhKh_BSE:15}).
 In particular, there holds 
\begin{equation} \label{eqn:L_V_factor}
 V \approx L_V L_V^T,\quad L_V\in \mathbb{R}^{n \times R_V}, \quad R_V \leq R_B,
\end{equation}
with the rank estimates $R_V =R_V(\varepsilon) =\mathcal{O}(N_b |\log \varepsilon |)$, and 
$\mathop{\mathrm{rank}}(\widetilde{W})\leq \mathop{\mathrm{rank}}(V)$. 

In \cite{BDKK_BSE:16}, it was shown that the matrix $\widehat{W}$, which does not exhibit 
an accurate low rank representation,
can be well approximated by a block diagonal matrix
\[
 \widehat{W} \approx \mbox{blockdiag}[\widehat{B},D],
\]
where $\widehat{B}$ is a $N_W\times N_W$ dense block with $N_W=O(n^\alpha)$, $\alpha<1$. 
The size of $N_W$ is nearly the same as the rank parameter of $L_V$.  
As a result, the TDA matrix $A$ can be approximated by a sum of a block-diagonal matrix and 
a low rank matrix shown in Figure \ref{fig:Matr_AN}, i.e.,
\[
 A   \approx  \widehat{A}= \boldsymbol{\Delta \varepsilon}+ Q Q^T - \mbox{blockdiag}[\widehat{B},D]
 \equiv \mbox{blockdiag}[{B}_0,D_0] + Q Q^T.
\]

\begin{figure}[htbp]
\centering
\includegraphics[width=11.0cm]{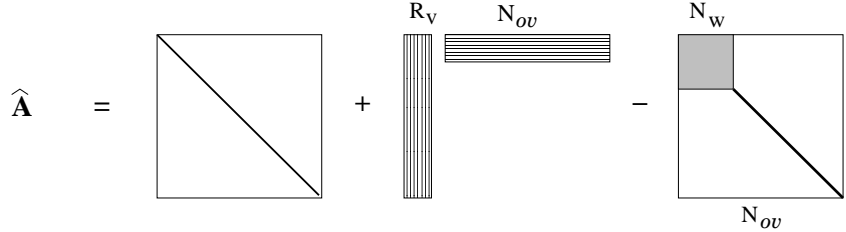} 
\caption{\small Diagonal plus low-rank plus reduced-block structure of the matrix $\widehat{A}$.}
\label{fig:Matr_AN}
\end{figure}

An efficient structured solver designed to calculate a number of minimal eigenvalues of 
the block-diagonal plus low-rank representation of the BSE/TDA matrices 
is described in \cite{BDKK_BSE:16}.  It is based on an efficient subspace iteration  
of the matrix inverse, which for rank-structured matrix formats can be evaluated 
efficiently by using the Sherman-Morrison-Woodbury formula,
thus reducing the numerical expense of the direct diagonalization down to $\mathcal{O}(N_b^2)$ 
in the size of the atomic orbitals basis set, $N_b$. 
  Furthermore, this solver  also includes a QTT-based compression scheme,  
 where both eigenvectors and the rank-structured BSE matrix blocks are 
 represented by block-QTT tensors. The block-QTT representation of the eigenvector 
 is determined by an alternating least squares (ALS) iterative algorithm.
The overall asymptotic complexity for computing several smallest in modulo eigenvalues in
the BSE spectral problem by using the QTT approximation is estimated  to be  
$\mathcal{O}(\log(N_o) N_o^{2})$, where $N_o$ is the number of
occupied orbitals. 

 Matrices in the form (\ref{eqn:BSE-GW1}) are called  $J$-symmetric or Hamiltonian, see 
\cite{BeFaYa:15} for implications on the algebraic properties of the BSE matrix.
In particular, solutions of equation (\ref{eqn:BSE-GW1}) come in pairs: excitation energies $\omega_k$
with eigenvectors $({\bf x}_k,{\bf y}_k)$, and de-excitation energies
$-\omega_k$ with eigenvectors $({\bf y}_k^\ast,{\bf x}_k^\ast)$.

The simplification in the  BSH, $H$, defined by the $n \times n$ symmetric diagonal  block $A$
is called the Tamm-Dancoff (TDA) approximation. In what follows, we are interested in
the TDA spectral problem, 
$$
A {\bf u}_k = \lambda_k {\bf u}_k,\quad k=1,\ldots,n,
$$ 
providing good approximations to $\omega_k,{\bf x}_k$.

 In general, methods for solving partial eigenvalue problems for matrices with a special structure 
as in the BSE setting are conceptually related to the 
approaches for Hamiltonian matrices \cite{BeFa:97,BeMeXu:98,FaKre:06,BGFa:15},
particularly to those based on minimization principles \cite{BaiLi:12,BaiLi:13}.
A structured Lanczos algorithm for estimation of the optical absorption spectrum was described in 
\cite{ShJoLiYaDeLo:16}.
Various structured eigensolvers tailored for electronic structure calculations are discussed 
in \cite{RoGeSaBa:08,RoLuGa:10,DeSaStJaCoLo:12,NaPoSaad:13,LinSaadYa:15,ShJoYaDeLo:16}.

\subsection{Density of states for symmetric matrices}\label{ssec:DOS_def}

To fix the idea, we first consider the case of symmetric matrices.
Following \cite{LinSaadYa:15}, we use the simple definition of the DOS for symmetric matrices
\begin{equation} \label{eqn:DOS}
 \phi(t)= \frac{1}{n}\sum\limits_{j=1}^{n} \delta(t-\lambda_j),\quad t,\lambda_j\in [0,a],
\end{equation}
where $\delta$ is the Dirac   distribution and the $\lambda_j$'s are the eigenvalues 
of $A=A^T$  ordered as  $\lambda_1 \leq \lambda_2 \leq \cdots \leq \lambda_n$.

Several classes of blurring approximations to $\phi(t)$ are used in the literature.
One can replace each Dirac-$\delta$ by a Gaussian function with width $\eta>0$, i.e.,
\[
\delta(t) \rightsquigarrow
 g_\eta(t)= \frac{1}{\sqrt{2 \pi}\eta}\exp{\left(-\frac{t^2}{2 \eta^2}\right)},
\]
where the choice of the regularization parameter $\eta$ depends on the particular problem setting.
As a result, (\ref{eqn:DOS}) can be approximated by 
\begin{equation} \label{eqn:DOS_gauss}
 \phi(t)\approx \phi_\eta(t):= \frac{1}{n} \sum\limits_{j=1}^{n} g_\eta(t -\lambda_j),
\end{equation}
on the whole energy interval $[0,a]$. 

We may also replace each Dirac-$\delta$ by a Lorentzian, i.e.,
\begin{equation} \label{eqn:Delta_Lorentz}
\delta(t) \rightsquigarrow
 L_\eta(t):= \frac{1}{\pi} \frac{\eta}{t^2 + \eta^2} = 
 \frac{1}{\pi} \mbox{Im}\left( \frac{1}{t- i \eta }\right),
\end{equation}
so that an approximate DOS can be written as
\begin{equation} \label{eqn:DOS_Lorentz}
 \phi(t)\approx \phi_\eta(t):= \frac{1}{n} \sum\limits_{j=1}^{n} L_\eta(t -\lambda_j).
\end{equation}
When $\eta\to 0_+$, both Gaussians and Lorentzians converge to the 
Dirac distribution, i.e.,
\[
 \lim\limits_{\eta\to 0_+} g_\eta(t) = \lim\limits_{\eta\to 0_+} L_\eta(t)=\delta(t).
\]
However, they exhibit different features of the approximant for small $\eta >0$.
In the case of Gaussians, one expects  a  sharp resolution of the spectral peaks, while 
the Lorentzian based representation aims to resolve better the global landscape of $\phi(t)$.

Both functions $\phi_\eta(t)$ and $L_\eta(t)$ are continuous, 
hence, they can be discretized by sampling on  a  fine grid $\Omega_h$ over $[0,a]$. 
In the following, we use the uniform cell-centered $N$-point grid with 
the mesh size $h=a/N$.

In what follows, we focus on the case of Lorentzian blurring, which will be motivated
later on, and apply it to the TDA approximation of the BSE problem 
(see \S \ref{ssec:BSE_setting} below). 
We use the simplified block-diagonal 
plus low-rank approximation to the matrix $A$, see \cite{BeKhKh_BSE:15,BDKK_BSE:16}, 
which allows efficient explicit representation of the shifted inverse matrix.

The numerical illustrations in \S \ref{ssec:DOS_def} represent the DOS
for the  H$_2$O molecule and H$_2$ chains  broadened by 
Gaussians (\ref{eqn:DOS_gauss}). The data corresponds to the 
reduced basis approach via rank-structured approximation applied to the symmetric
TDA model \cite{BeKhKh_BSE:15,BDKK_BSE:16} described by the matrix block $A$ of 
the full BSE system matrix. 

\begin{figure}[htb]
\centering
\includegraphics[width=7.0cm]{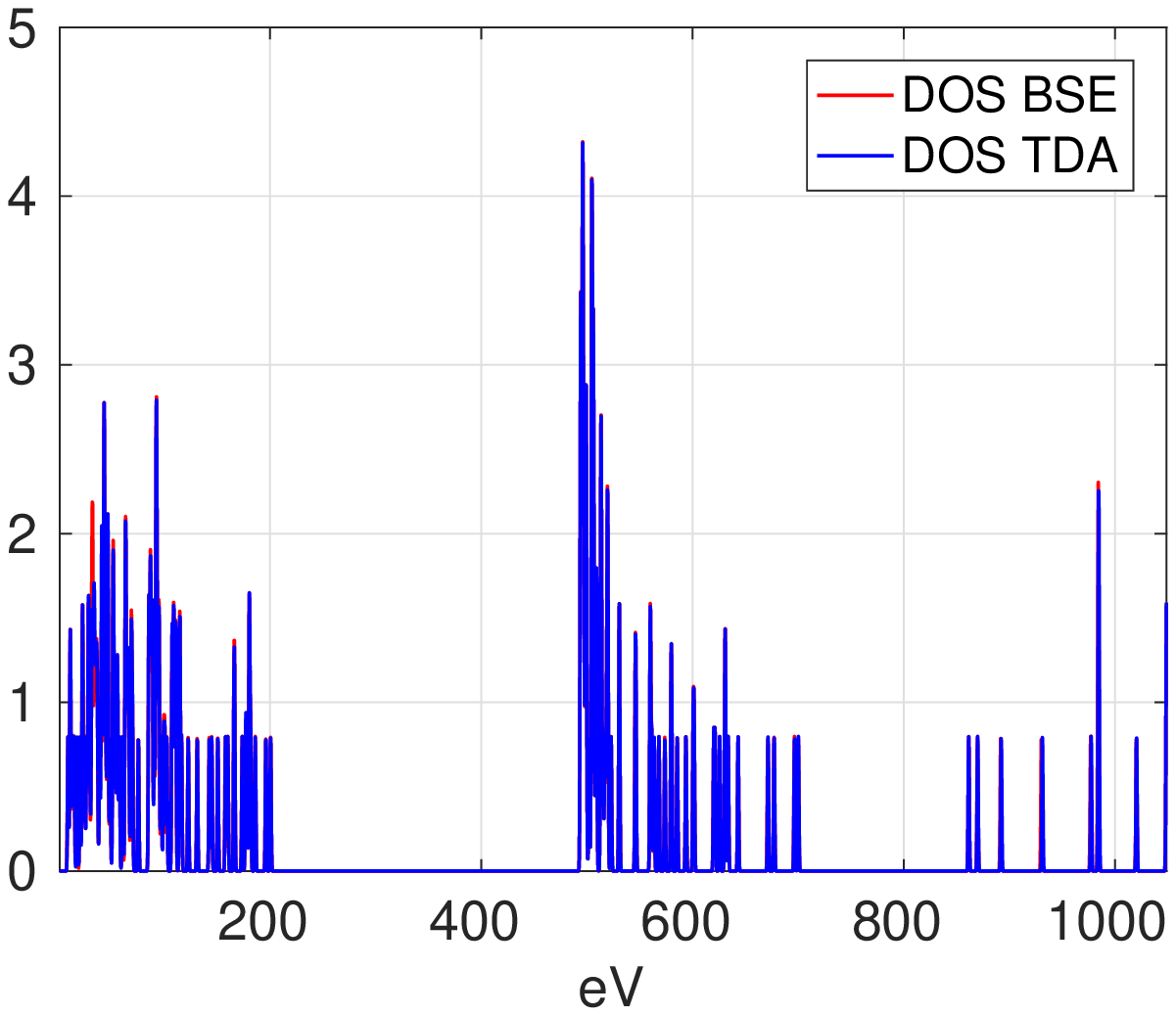} 
\includegraphics[width=7.0cm]{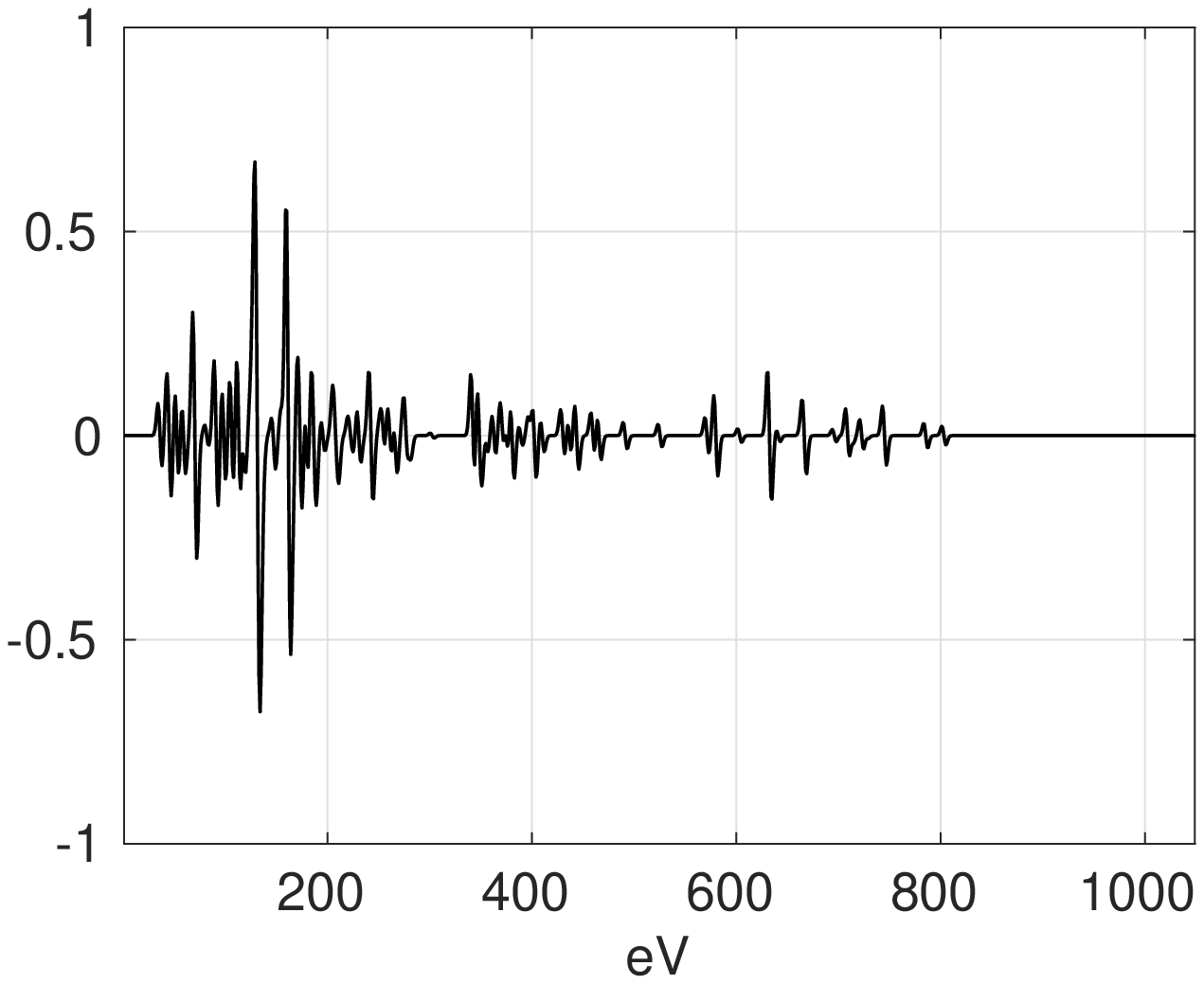}
\caption{\small DOS for H$_2$O, $\eta=0.5$: exact BSE vs. TDA on the full spectrum (left), 
the absolute error (right). }
\label{fig:BSE_vs_TDA_H2O}
\end{figure}

It was numerically demonstrated in \cite{BeKhKh_BSE:15} that
the spectrum of the TDA model provides a good approximation to the spectrum of the full BSE Hamiltonian. 
The difference between the two is on the order of $10^{-2}$ for molecules of moderate size.

Figure \ref{fig:BSE_vs_TDA_H2O}, left, compares the DOS for the H$_2$O molecule 
calculated via the eigenvalues of the full BSE Hamiltonian and those of the TDA approximation,
while on the right we display the corresponding maximum error.

Figure \ref{fig:BSE_vs_TDA_H2O_zoom}, left, compares the same DOS calculations but 
zoomed on the first compact energy interval $[0,40]$ eV. 
The red curve corresponds to the full BSE data, and the blue one represents the TDA case.
The figure on the right displays  the corresponding maximum error.  
   \begin{figure}[htb]
\centering
\includegraphics[width=7.0cm]{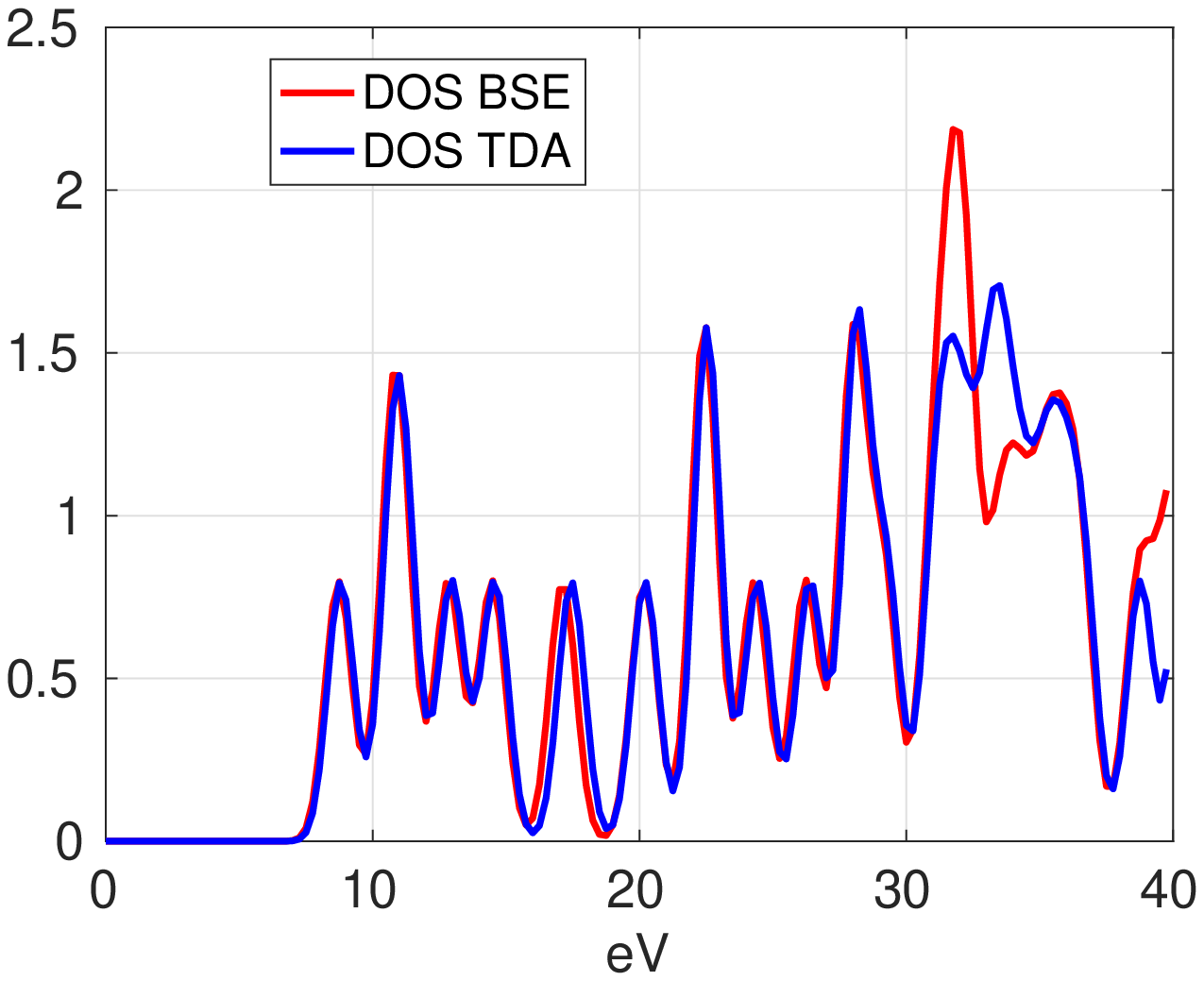} 
\includegraphics[width=7.0cm]{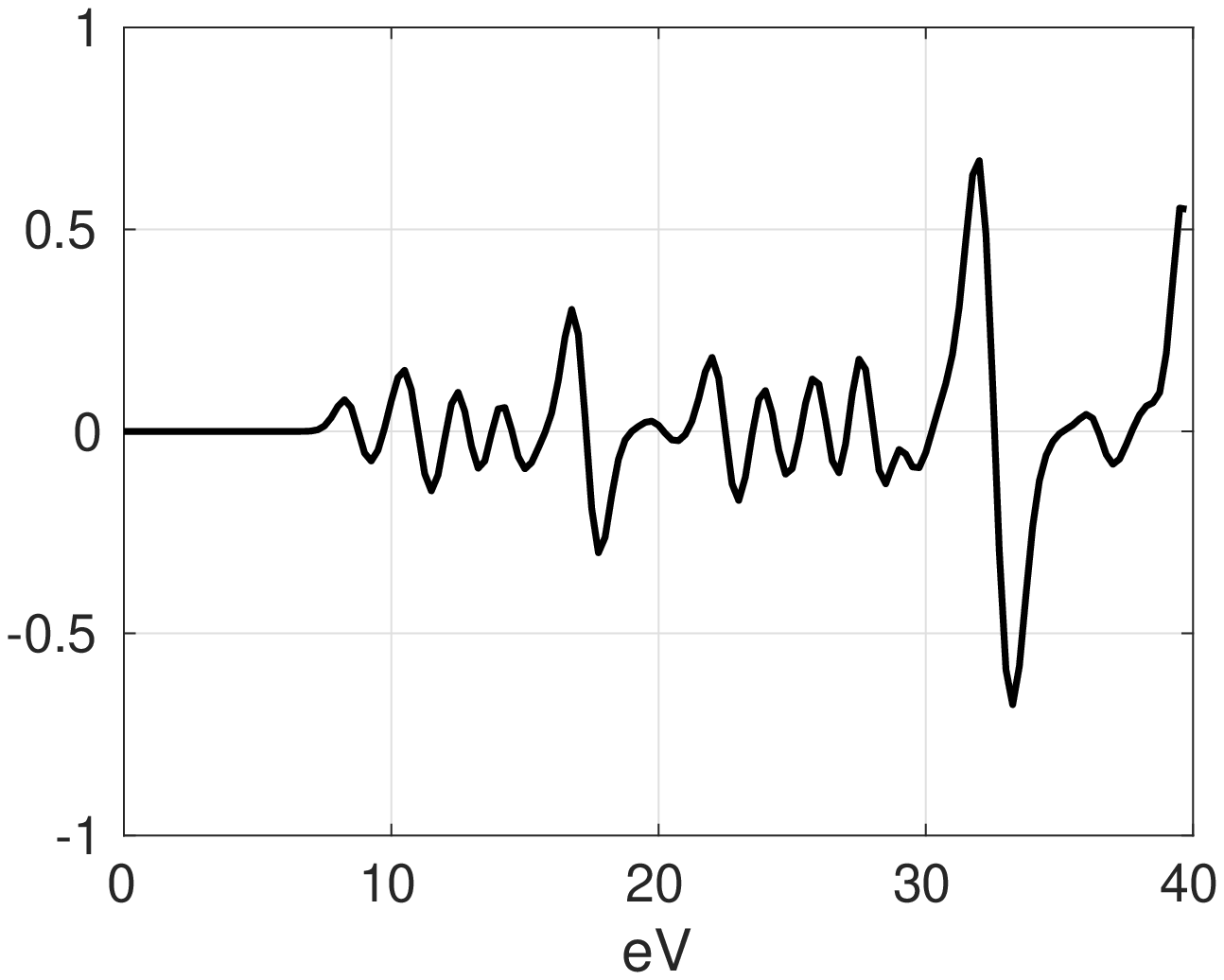}
\caption{\small DOS for H$_2$O on the energy sub-interval $[0,40]$: exact BSE vs. TDA (left), 
and the error (right). }
\label{fig:BSE_vs_TDA_H2O_zoom}
\end{figure}

Figure \ref{fig:Simp_vs_TDA_H2O}, left, represents the DOS for H$_2$O 
computed by using the exact TDA spectrum (blue) 
and its approximation based on a simplified model obtained via low-rank approximation to $A$
(red), while the right figure shows the relative error.

Figures \ref{fig:H_lattice} presents the DOS for   H$_{16}$ (left) and H$_{32}$ (right) 
chains of Hydrogen atoms.
We observe the essential similarity in the shapes (only the amplitude is changing) which is 
apparently a consequence of quasi-periodicity of the system. 
\begin{figure}[htb]
\centering
\includegraphics[width=7.0cm]{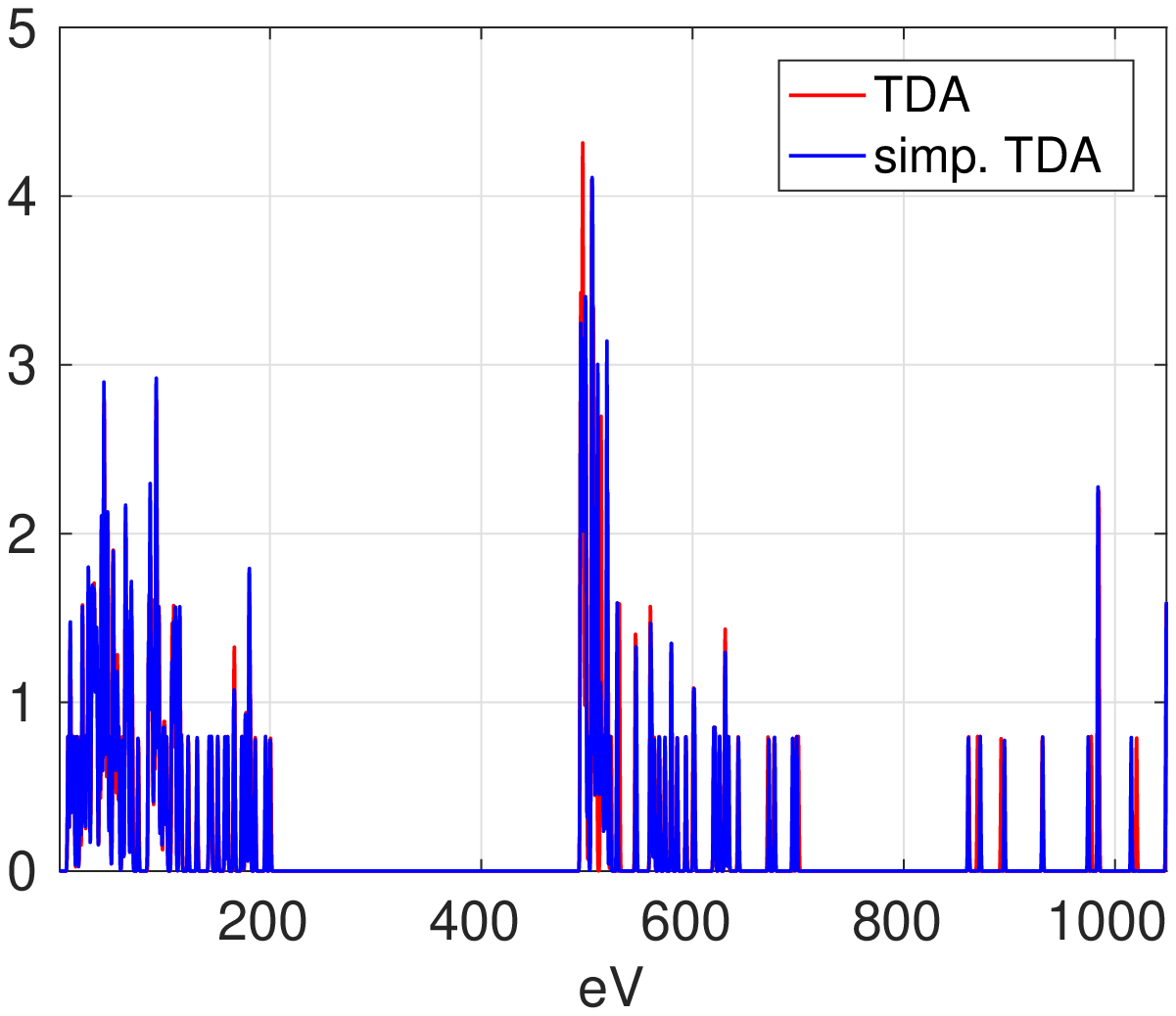} 
\includegraphics[width=7.0cm]{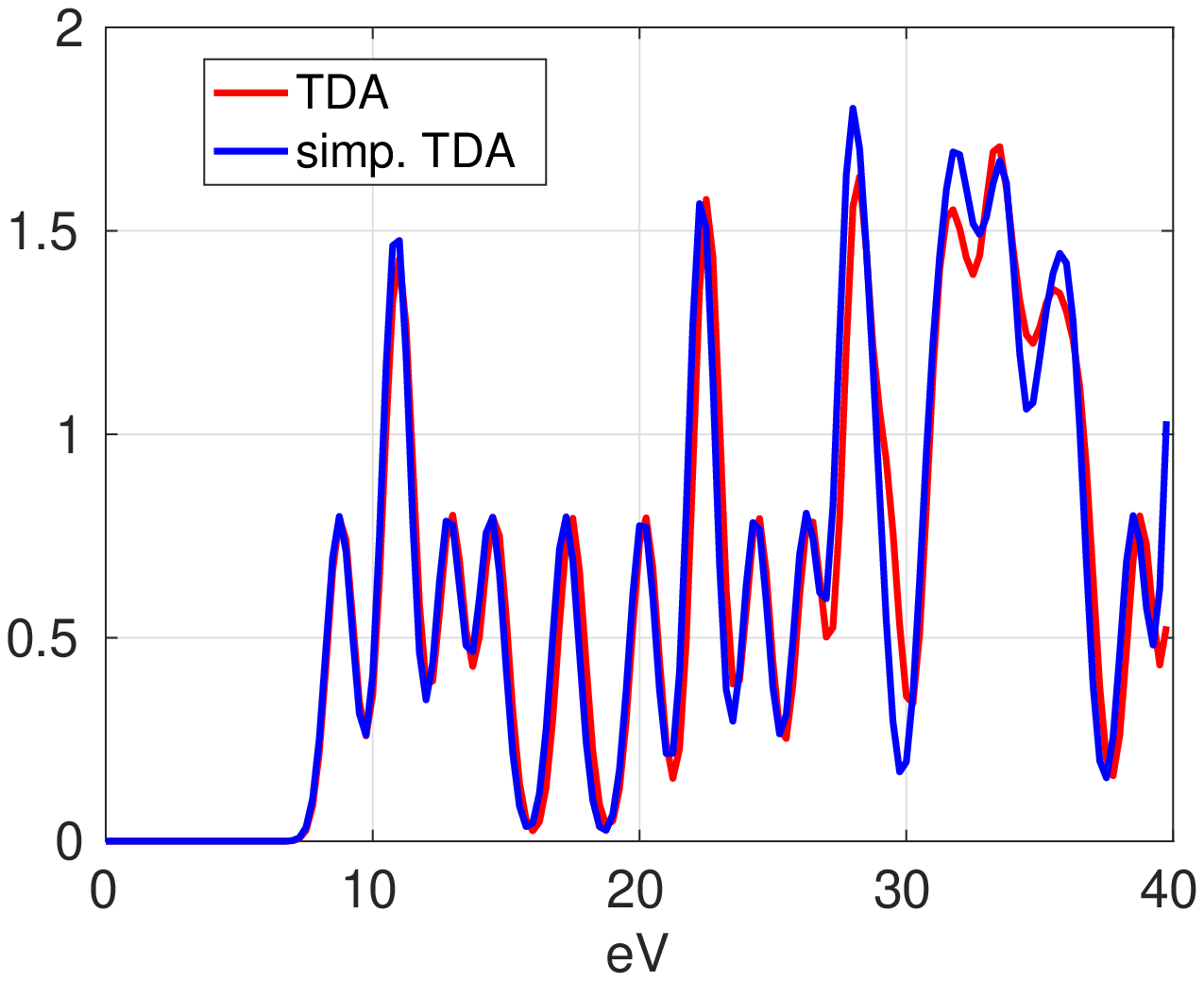}
\caption{\small DOS for H$_2$O. Exact TDA vs. simplified TDA (left), 
zoom of the small spectral interval (right).}
\label{fig:Simp_vs_TDA_H2O}
\end{figure}

\begin{figure}[htb]
\centering
\includegraphics[width=7.0cm]{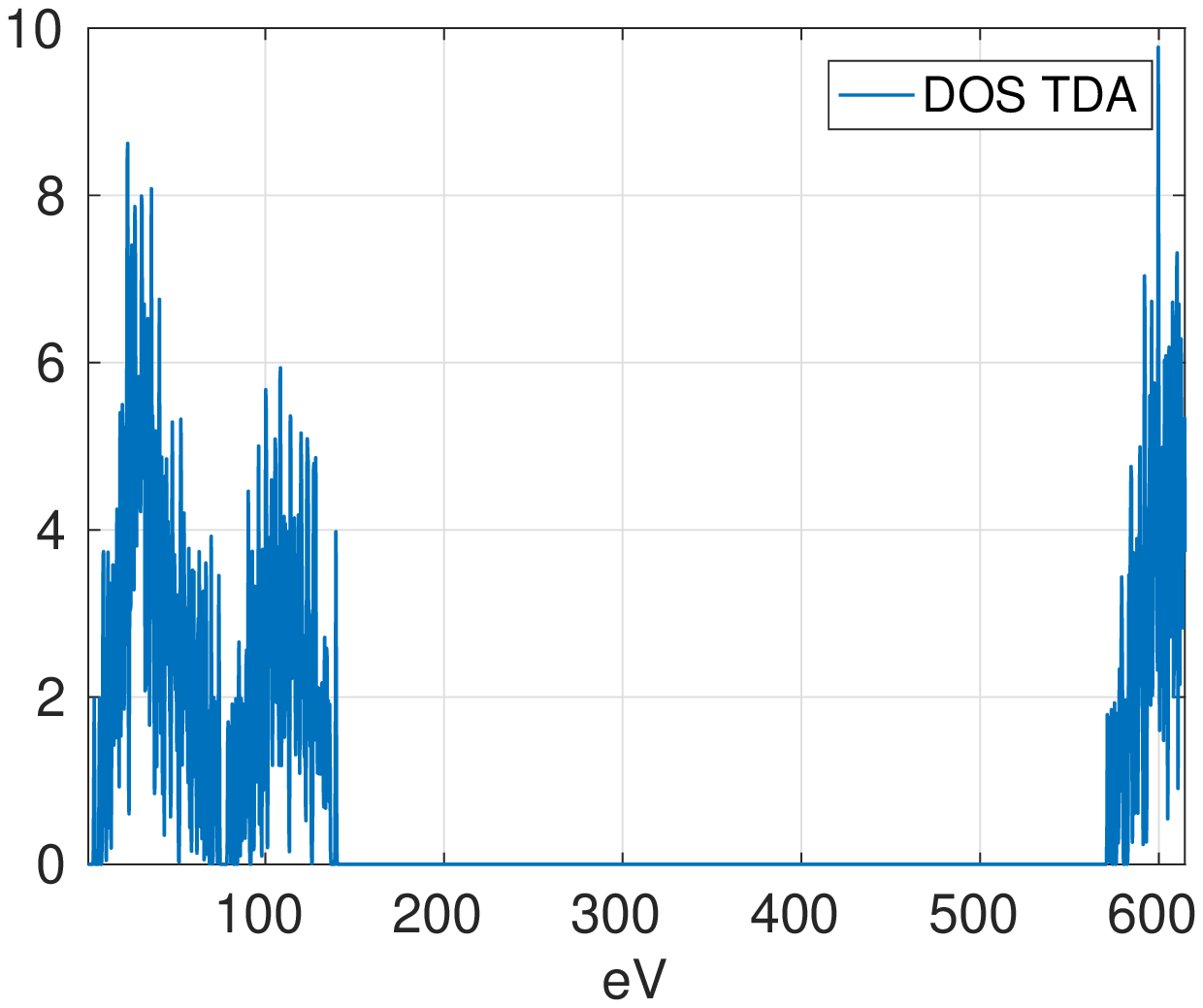} 
\includegraphics[width=7.0cm]{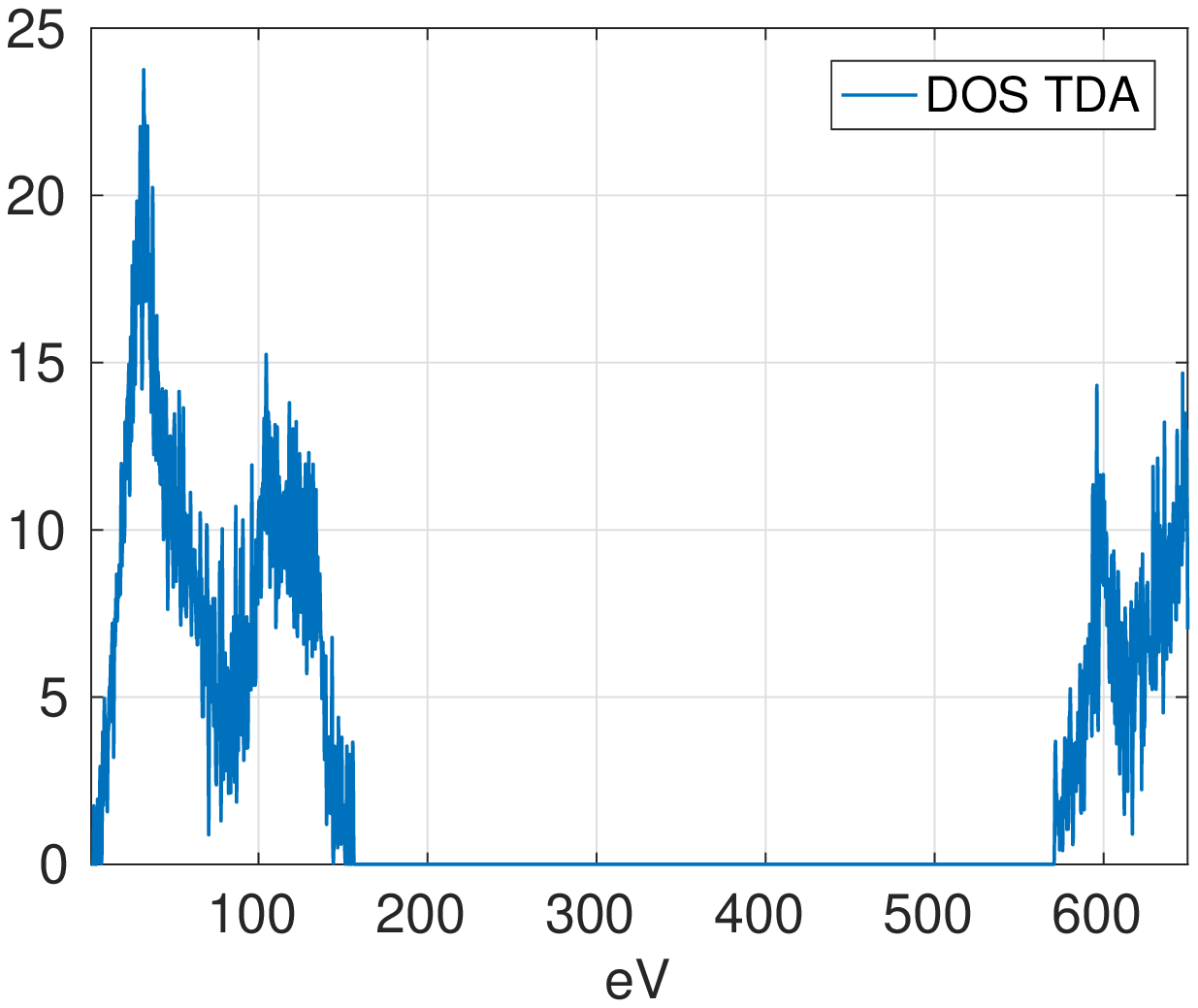}
\caption{\small DOS for H$_{16}$ (left) and H$_{32}$ (right) chains of Hydrogen atoms.}
\label{fig:H_lattice}
\end{figure}
The rank-structured approach to calculation of the molecular absorption spectrum in 
the case of full BSE is sketched in \S\ref{sec:BSE_case}. This topic
will be addressed elsewhere.

\subsection{General description of the existing computational schemes}

One of the commonly used approaches to the numerical approximation of both 
functions $g_\eta(t)$ and $L_\eta(t)$ 
is based on the construction of   certain polynomial or fractional polynomial 
interpolants whose evaluation at each sampling point $t_k$ requires the solution 
of  a  large linear system with the BSE/TDA matrix, i.e., remains expensive. 

In the case of Lorentzian broadening (\ref{eqn:DOS_Lorentz}) the regularized  DOS
takes  the  form 
\begin{equation} \label{eqn:DOS_Loren_ImTr}
\phi(t)\approx \phi_\eta(t):= 
 \frac{1}{n\pi} \sum\limits_{j=1}^{n}\mbox{Im}\left( \frac{1}{(t-\lambda_j) - i \eta }\right)=
 \frac{1}{n\pi} \mbox{Im}\, \mbox{Trace}[(tI-A - i \eta I)^{-1}].
\end{equation}
To keep real-valued arithmetics, likewise, we can write the latter in the form
\begin{equation} \label{eqn:DOS_Loren_RTr}
 \phi_\eta(t):=\frac{1}{n\pi} \sum\limits_{j=1}^{n}\frac{\eta}{(t-\lambda_j)^2 + \eta^2} =
 \frac{1}{n\pi}\mbox{Trace}[ ( (t I - A)^2 + \eta^2 I)^{-1}].
\end{equation}
In both cases the task of  computing  the approximate DOS $\phi_\eta(t)$ reduces to approximating 
the trace of the matrix resolvent 
$$
(tI-A - i \eta I)^{-1}\quad  \mbox{or} \quad ((t I - A)^2 + \eta^2 I)^{-1}.
$$
Here,  the price to pay for real-valued arithmetics is to 
address the more complicated low-rank structure in $(t I - A)^2$. 

The traditional approach \cite{LinSaadYa:15} 
to approximately computing the traces of the matrix-valued analytic function $f(A)$ reduces
this task to the estimation of the mean of $v_m^T f(A) v_m$ over a sequence of random 
vectors  $v_m$, $m=1,\ldots,m_r$,   satisfying certain condition (see \cite{LinSaadYa:15}, Theorem 3.1). 
That is, $\mbox{Trace}[f(A)]$ is approximated by
\begin{equation} \label{eqn:DOS_Trace_f(A)}
 \mbox{Trace}[f(A)]\approx \frac{1}{m_r}\sum\limits_{m=1}^{m_r}v_m^T f(A) v_m.
\end{equation}
The calculation of (\ref{eqn:DOS_Trace_f(A)}) for 
\begin{equation} \label{eqn:ReIm_resol}
f_1(A)=(tI-A - i \eta I)^{-1} \quad  \mbox{or} \quad f_2(A)= ((t I - A)^2 + \eta^2 I)^{-1}
\end{equation}
reduces to solving  linear  systems in the form of
\begin{equation} \label{eqn:DOS_Trace_syst_compl}
(tI - i \eta I -A)x= v_m \quad  \mbox{for} \quad m=1,\ldots,m_r,
\end{equation}
or 
\begin{equation} \label{eqn:DOS_Trace_syst_real}
 (\eta^2 I+(tI-A)^2 )x= v_m \quad  \mbox{for} \quad m=1,\ldots,m_r.
\end{equation}
These linear systems need to be solved for many target points $t=t_k\in [a,b]$ in 
the course of a chosen interpolation scheme. 

In the case of rank-structured matrices $A$, the solution of 
equations (\ref{eqn:DOS_Trace_syst_compl}) or (\ref{eqn:DOS_Trace_syst_real}) 
can be implemented with a lower cost.
However, even in this favorable situation one requires a relatively large
number $m_r$ of stochastic realizations to obtain satisfactory mean value approximation.  
The convergence rate is expected to be on the order of $O(1/\sqrt{m_r})$.
On the other hand, with the limited number of interpolation points, 
the polynomial type of interpolation schemes applied to highly
non-regular shapes as shown, say, in Figure \ref{fig:Simp_vs_TDA_H2O} (left),
can  only provide limited resolution and is unlikely to reveal spectral gaps and 
many local peaks of interest.

\section{Fast evaluation of DOS for rank-structured matrices}
\label{sec:Rstruc_DOS}

\subsection{DOS by the trace of rank-structured matrix inverse}
\label{ssec:Rstruc_DOS_M_inv}

In what follows, we propose an approach that is based on evaluating the trace term
in \eqref{eqn:DOS_Loren_ImTr} directly (without stochastic sampling). 
This approach relies on the following two techniques:
\begin{itemize}
 \item[(A)]  using the low-rank BSE matrix structure as in \cite{BeKhKh_BSE:15},
 which allows for each fixed $t\in [0,a]$ the direct matrix inversion and computation 
 of the respective traces,
 
 \item[(B)] the low-rank QTT tensor interpolation of the function $L_\eta(t)$ sampled 
on a fine uniform grid $\{t_1,\ldots,t_M\}$ in the whole spectral interval $[0,a]$ 
or on some subinterval of $[0,a]$.
\end{itemize}

For the class of block-diagonal plus low-rank matrices
arising in the reduced model approach for BSE problem \cite{BeKhKh_BSE:15,BDKK_BSE:16},
we have (see \S\ref{ssec:BSE_setting} for more details)
\begin{equation} \label{eqn:Block_A}
 A = E + P Q^T, \quad \mbox{with} \quad P,Q\in \mathbb{R}^{n\times R},
 \quad  E =\mbox{blockdiag}\{B_0,D_0\},
\end{equation}
where the rank parameter $R$ is small compared to $n$, the full $n_B\times n_B$ matrix block $B_0$ 
is of size $n_B=O(n^\alpha)$, $0<\alpha < 1$, and $D_0$ is a diagonal matrix of size $n- n_B$. 

Notice that even in the case of structured matrices in (\ref{eqn:Block_A}) the traditional 
approach by  (\ref{eqn:DOS_Trace_f(A)}) leads to a sequence of
linear systems  (\ref{eqn:DOS_Trace_syst_compl}) to be solved 
many times in the course of stochastic sampling,
for each of many interpolation points $t\in [0,a]$.
 
In our approach, for the class of rank-structured matrices (\ref{eqn:Block_A}),
we propose to avoid stochastic sampling  in (\ref{eqn:DOS_Trace_f(A)})  by
introducing a direct scheme that allows us to evaluate the trace of matrices
$f_1(A)$ or $f_2(A)$ defined in (\ref{eqn:ReIm_resol}),
corresponding to the matrix resolvent in (\ref{eqn:DOS_Loren_ImTr}) and (\ref{eqn:DOS_Loren_RTr}),
respectively, by one-step straightforward matrix calculation.

To that end, let us first construct the reduced-model approximation to the matrix 
inverse $A^{-1}$ for the matrix in (\ref{eqn:Block_A}), 
where the block-diagonal part  $E(t)=\mbox{blockdiag}\{B(t),D(t)\}$  corresponds to
the case of (\ref{eqn:DOS_Loren_ImTr}), i.e., 
\begin{equation} \label{eqn:BlockBD}
 B(t)  = tI_B - i \eta I_B + B_0, \quad   D(t)   =tI_D - i \eta I_D +D_0.
\end{equation}
Here $B_0$ and $D_0$ denote the 
corresponding matrix blocks in the representation of the diagonal block $A$ in the initial 
BSE matrix, see (\ref{eqn:Block_A}), and $I_B,I_D$ denote the identity matrices 
 corresponding to  the respective index subsets. For the ease of exposition,
we further assume that the matrix size of the block $B$ in (\ref{eqn:BlockBD}) 
is bounded by $n_B=O(n^\alpha)$ with $\alpha\leq 1/3$. 
This assumption on the block size ensures the linear complexity scaling of our algorithm in the
matrix size $n $. 

In what follows, we use the notion ${\bf 1}_m$ for a length-$m$ 
vector of all ones, and $\odot$ for the Hadamard product of 
matrices.

The following  result asserts that the cost of trace calculations 
is estimated to be $O(n R)$. 
\begin{theorem}\label{thm:Trace_cost}
  Let the matrix family $A=A(t)$,   $t\in [0,a]$,  be given by (\ref{eqn:Block_A}), with 
 $$
 E=E(t)=\mathrm{blockdiag}\{B(t),D(t)\},
 $$
 where $B(t),D(t)$ are defined in (\ref{eqn:BlockBD}).
  Then the trace of the matrix inverse $A(t)^{-1}$ can be calculated explicitly by 
 \[
  \mathrm{trace}[A(t)^{-1}]= \mathrm{trace}[B(t)^{-1}] + \mathrm{trace}[D(t)^{-1}] - 
  {\bf 1}_n^T ( {U(t)} \odot {V(t)}) {\bf 1}_R,
 \]
 where ${U(t)}=E(t)^{-1}P K(t)^{-1}\in \mathbb{R}^{n\times R}$,   
 ${V(t)}= E(t)^{-1}Q\in \mathbb{R}^{n\times R}$,
 and 
 $$
 K(t)=I_R+ Q^T  E(t)^{-1}(t) P $$ 
 is a small $R \times R$ matrix.
 For fixed $t\in [0,a]$, assume that $n_B=O(n^\alpha)$ with $\alpha\leq 1/3$, 
 then the numerical cost is estimated by $O(nR^2)$.
 \end{theorem}
\begin{proof}
The analysis relies on the particular structure   of the matrix blocks.
Indeed, we use the direct trace representation for both rank-$R$ and block-diagonal matrices.
Our argument is based on the observation that the trace of a rank-$R$ matrix
$U(t) V(t)^T$, where $U(t),V(t)\in \mathbb{R}^{n\times R}$, $U(t)=[{\bf u}_1,\ldots , {\bf u}_R] $,
$V(t)=[{\bf v}_1,\ldots , {\bf v}_R] $, ${\bf u}_k, {\bf v}_k  \in \mathbb{R}^{n}$, 
can be calculated in terms of skeleton vectors by 
\[
 \mbox{trace}[U(t) V(t)^T]=\sum_{k=1}^R \langle {\bf u}_k, {\bf v}_k \rangle = 
 {\bf 1}_n^T ( U(t) \odot V(t)) {\bf 1}_R,   
\]
at the expense $O(R n)$.
For fixed $t$, define the rank-$R$ matrices  by  
\[
 U(t)=E(t)^{-1}P K(t)^{-1}, \quad V(t)= E(t)^{-1}Q,
\]
then the Sherman-Morrison scheme leads to the representation, 
see \cite{BDKK_BSE:16},
 \[
  A(t)^{-1}= \mbox{blockdiag}\{ B(t)^{-1},D(t)^{-1} \} - E(t)^{-1}P K(t)^{-1} Q^T  E(t)^{-1},
 \]
 where the last term simplifies to
 \[
  E(t)^{-1}P K(t)^{-1} Q^T  E(t)^{-1} = U(t) V(t)^T.
 \] 
Now we apply the above formula for the trace of a rank-$R$ matrix to obtain the 
desired representation.

The complexity estimate follows taking into account the bound on the size of matrix 
block $B$. Indeed, forming $U(t)$ involves solving the linear system 
$P_1(t) = U(t) K(t)$, for $U(t)$, where $P_1(t)$ is the pre-computed $E(t)^{-1}P$, which can be computed by
assumptions at the cost $O(nR)$. Here $P_1(t)$ would be re-used to compute $K(t)$ itself, and thus stored.
The cost for solving this system of equations is $2/3R^3$ (LU factorization of $K(t)$), 
plus $2nR^2$ for backward/forward solves. This completes the proof.
\end{proof}

The above representation has to be applied many times for 
calculating the trace of $E(t_m)^{-1}P K(t_m)^{-1} Q^T  E(t_m)^{-1}$  
at each fixed interpolating point $t_m$, $m=1,\ldots,M$.

Here, we notice that  the price to pay   for the real arithmetics in equation
(\ref{eqn:DOS_Trace_syst_real})  is that we compute  with squared matrices 
which,  however,    do not increase the asymptotic complexity since there is no
increase of the rank in the rank-structured representation of the system matrix, 
see the following Theorem \ref{thm:Trace_cost_real}.  
In our applications we do not expect a loss of numerical stability of the algorithm since
 the condition numbers of $E(t)$ are moderate.  In what follows we denote by
$[U,V]$ the concatenation of two matrices of compatible size.

\begin{theorem}\label{thm:Trace_cost_real}
 Given matrix $S=(tI-A)^2 + \eta^2 I$, where $A$ is defined by (\ref{eqn:Block_A}),
  then the trace of the real-valued matrix resolvent $S^{-1}(t)$ can be calculated explicitly by  
  \begin{equation} \label{eqn:Trace_realLR}
 \mathrm{trace}[S^{-1}]= \mathrm{trace}[E_{0}^{-1}] - 
 {\bf 1}_n^T (\overline{U} \odot \overline{V}) {\bf 1}_{2R},
\end{equation}
with $\overline{U}=E_0^{-1} \overline{P} K^{-1}\in \mathbb{R}^{n\times 2R}$, 
and $\overline{V} =  E_0^{-1} \overline{Q} \in \mathbb{R}^{n\times 2R} $,
where the real-valued block-diagonal matrix $E_0$ is given by
\[
 E_{0}(t)=\eta^2 I + t^2 I - 2t E + E^2=(\eta^2 + t^2) I 
+  \mathrm{blockdiag}[B^2- 2t B,D^2- 2t D],  
\]
and the rank-$2R$ matrices $\overline{P}, \overline{Q}$ are represented via concatenation
\[
 \overline{P}= [-2 tQ +E Q+  Q E + Q(Q^T Q),Q]\in \mathbb{R}^{n\times 2 R},\quad 
 \overline{Q}=[Q,EQ]  \in \mathbb{R}^{n\times 2 R}, 
\]
such that the small core matrix $K(t)\in \mathbb{R}^{2R \times 2R}$ takes the form 
$K(t)=I_R+ \overline{Q}^T  E_0^{-1}(t) \overline{P} $.

Assume that $n_B=O(n^\alpha)$ with $\alpha\leq 1/3$, 
 then the numerical cost is estimated by $O(n R^2)$ up to a low order term.
\end{theorem}
\begin{proof}
 Indeed, given the block-diagonal plus low-rank matrix $A$ in the form (\ref{eqn:Block_A}),
we obtain
\[
S=(tI-A)^2 + \eta^2 I = E_{0} + \overline{P}\, \overline{Q}^T,
\]
where the block-diagonal matrix $E_{0}$ and  the  rank-$2R$ matrix $\overline{P}\,\overline{Q}^T$ 
are defined as above.
Applying the Sherman-Morrison scheme as above to the block-diagonal plus rank-$2R$ matrix structure 
in $S$, the representation result follows.  Now we take into account that 
$$
\mathrm{trace}[E_{0}^{-1}]=\mathrm{trace}[(B^2- 2t B)^{-1}] + \mathrm{trace}[(D^2- 2t D)^{-1}],
$$ 
then the restriction on the size of the block $B$ proves the complexity bound similar 
to the argument in the proof of the previous theorem.  
\end{proof}

Based on  Theorems \ref{thm:Trace_cost} and \ref{thm:Trace_cost_real}, the calculations in  
item (A) can be implemented efficiently in both complex and real arithmetics. 
The following numerics demonstrates the efficiency of the DOS calculation 
for the rank-structured TDA matrix implemented in real arithmetics as
described by (\ref{eqn:Trace_realLR}) in Theorem \ref{thm:Trace_cost_real}.
\begin{table}[hbp]
 \begin{center}
 \begin{tabular}
[c]{|c|c|c|c|c|c|c|c|}%
\hline
Molecule & H$_2$O & NH$_3$ & H$_2$O$_2$ & N$_2$H$_4$ &C$_2$H$_5$OH & C$_2$H$_5$ NO$_2$ & C$_3$H$_7$ NO$_2$ \\
 \hline
 $n=N_{ov}$ & $180$ & $215$ & $531$ & $657$ & $1430$   & $3000$ & $4488$   \\
 \hline
Rank  $R$      &    $36$ & $30$ & $68$   & $54$ & $74$  & $129$  & $147$   \\
\hline
Total time $T$ (s)       & $6.7$ & $7.7$ & $33$ & $47$ & $219$   & $1084$ & $2223$   \\
\hline
 Scaled time $T/R^2$ (s) & $0.005$ & $0.008$ & $0.007$ & $0.017$ & $0.041$   & $0.065$ & $0.103$   \\
 \hline
\end{tabular}
\end{center}
\caption{\small Scaled times for the Algorithm in Theorem \ref{thm:Trace_cost_real}.}
\label{tab:ratio_NW2A}
\end{table} 

\begin{figure}[htb]
\centering
\includegraphics[width=7.0cm]{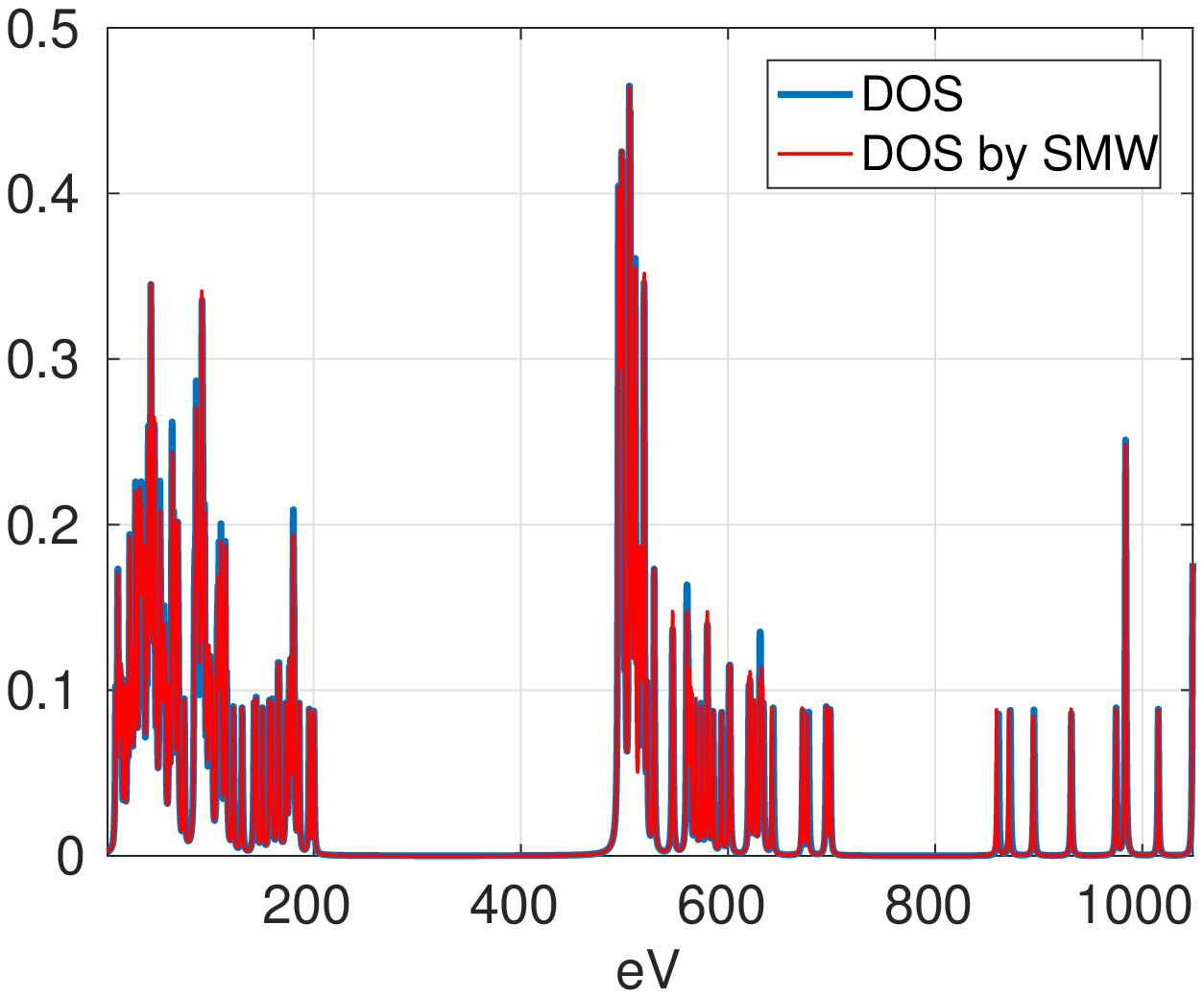}
 \includegraphics[width=7.0cm]{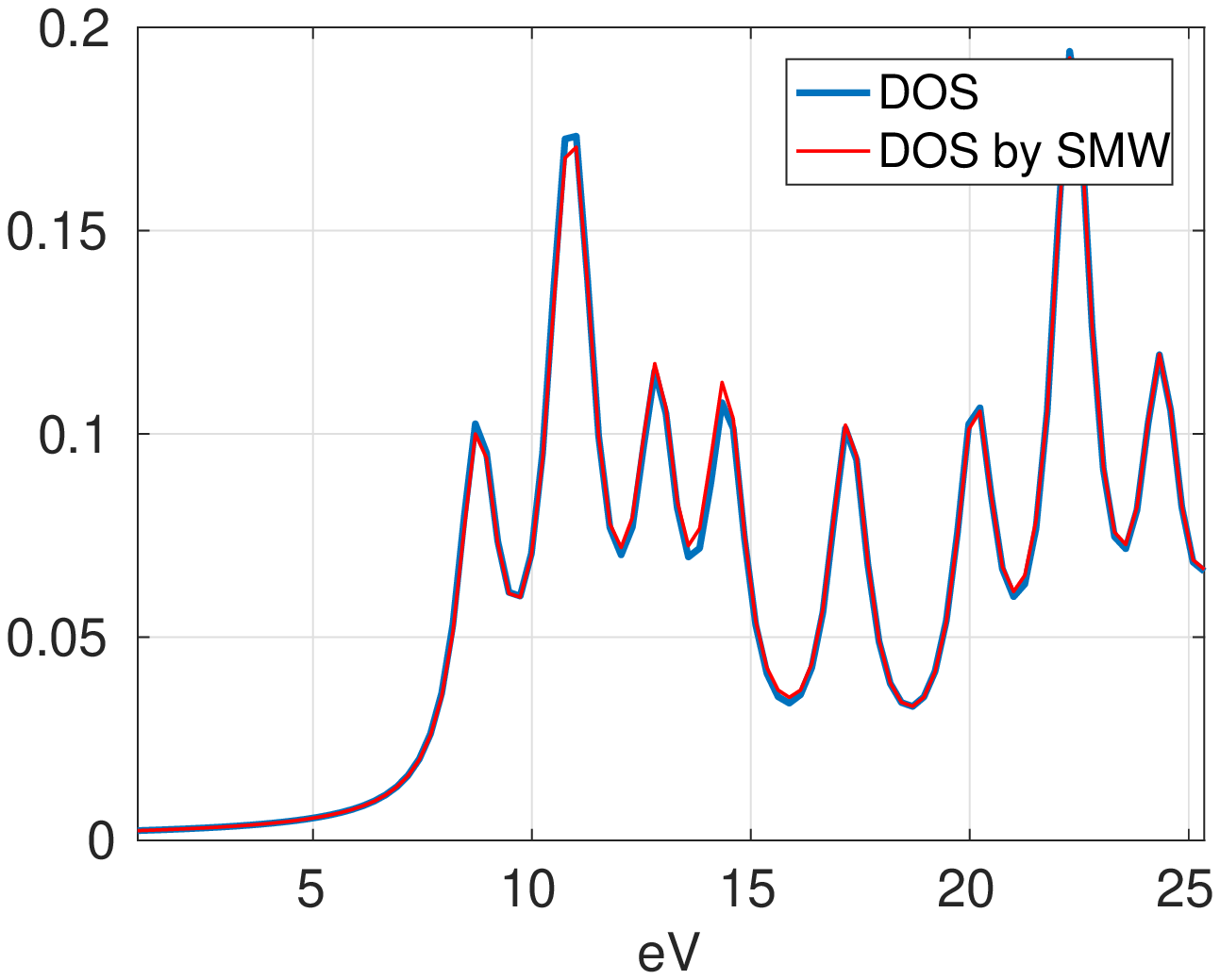}
\caption{\small Left: DOS for H$_2$O vs. its recovering by using the trace of 
matrix resolvents; Right: zoom on the small energy interval.}
\label{fig:DoS_TraceDirH2O}
\end{figure}
 \begin{figure}[htb]
\centering
\includegraphics[width=7.0cm]{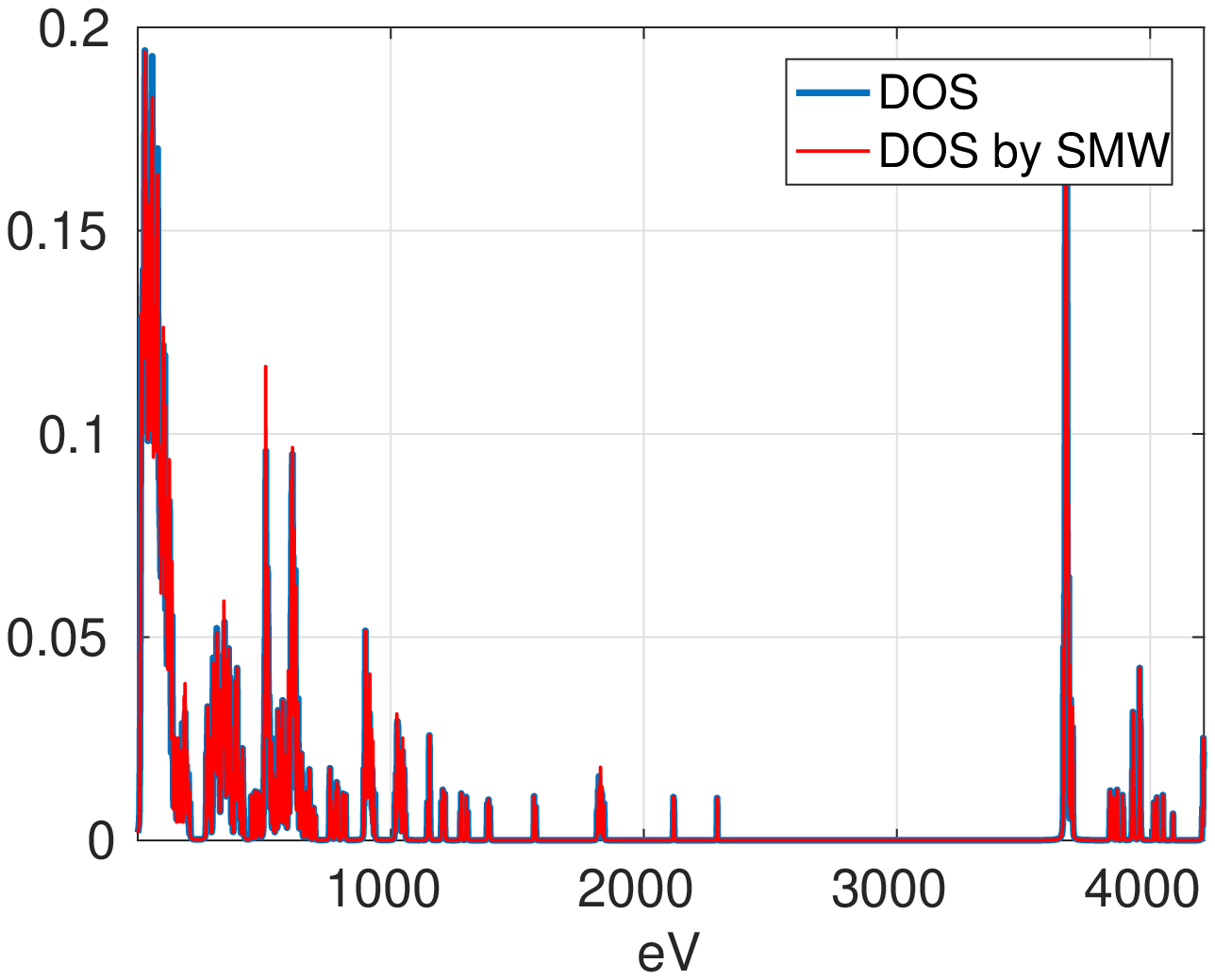}
 \includegraphics[width=7.0cm]{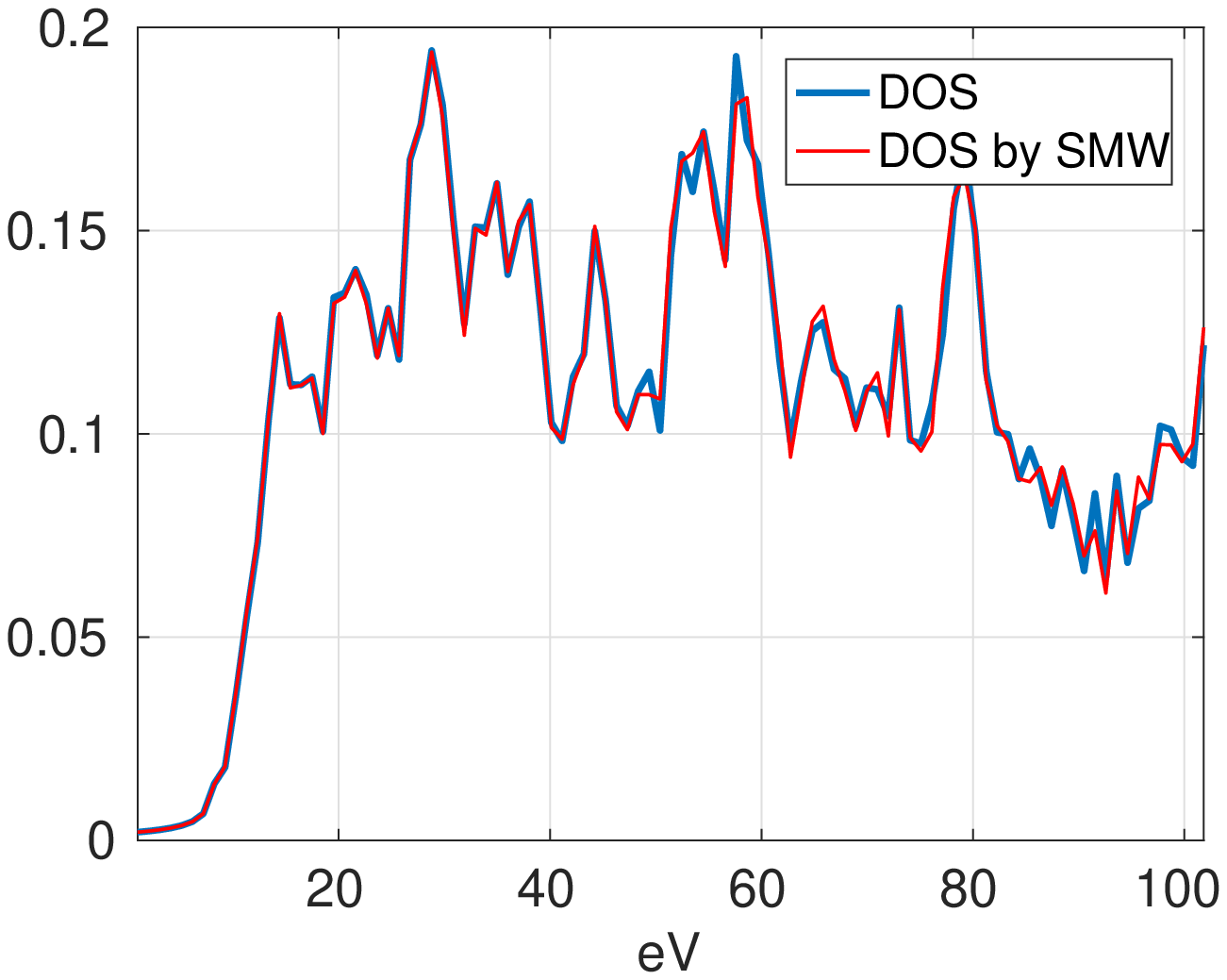}
 \caption{\small Left: DOS for Ethanol molecule vs. its recovering by using the trace of 
matrix resolvents; Right: zoom on the small energy interval.}
 \label{fig:DoS_TraceDirEtanol}
\end{figure}
 Figures \ref{fig:DoS_TraceDirH2O} and \ref{fig:DoS_TraceDirEtanol} demonstrate  
that using only the structure-based trace representation (\ref{eqn:Trace_realLR})
in Theorem \ref{thm:Trace_cost_real}, we obtain the approximation 
which resolves perfectly the DOS function (for the examples of H$_2$O and Ethanol molecules). 
The exact DOS is shown by the blue line, while the results of structure-based DOS calculation is 
indicated by the red line (we use the acronym ``SMW''  for the Sherman-Morrison-Woodbury
scheme). 
\begin{figure}[htb]
\centering
 \includegraphics[width=7.0cm]{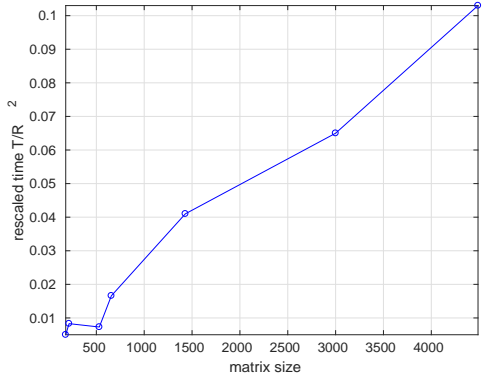}
\caption{\small Algorithm in Theorem \ref{thm:Trace_cost_real}: 
the  rescaled CPU time  $T/R^2$ versus $n$.}
\label{fig:DoS_trace_Times}
\end{figure}

Figure \ref{fig:DoS_trace_Times} shows  
the rescaled CPU time, i.e. $T/R^2$, where $T$ denotes the total CPU time for computing
the DOS by the algorithm  implied by Theorem \ref{thm:Trace_cost_real}.
This demonstrates almost linear complexity scaling of the algorithm in $n$, $O(R^2 n)$.
We applied the algorithm to  molecules of different system size 
$n$ (i.e. the size of TDA matrix) varying from $n=180$ till $n=4488$ 
(see Table \ref{tab:ratio_NW2A} for more details). 
In all cases the $N$-point representation grid with fixed $N=2^{14}$ was used.


We conclude that the algorithm based on  
representation (\ref{eqn:Trace_realLR}) demonstrates the perfect resolution of the DOS 
function at linear complexity in the system size which allows to treat  large  molecules.  

 The approach in item (B) requires fast trace calculations for many different values 
of parameter $t_m\in \tau=\{t_1,\ldots,t_M\} \subset [0,a]$ in the matrix resolvent. 
Finer resolution of the spectrum for large molecular systems leads to   a  considerable
increase of the number of samples $M$ that is practically equal to the grid size, $M=N$.
Hence, the total cost $O(M n R^2)$ may become prohibitively expensive since the trace 
computation for each fixed value of $t_m$ still requires complicated matrix operations 
(see  Theorems  \ref{thm:Trace_cost} and \ref{thm:Trace_cost_real}).

\subsection{ Calculating multiple traces of $A^{-1}$ with lower cost  }
\label{ssec:MultipleTr}

In this section, we describe a further enhancement scheme for fast multiple calculation of
traces on the large set of interpolation points.
We outline how it is possible to reduce the complexity of these calculations (reduced model)
by using  a  certain smoothness in $t$ in the parametric matrix resolvent by introducing the low rank
approximation of the large $n^2\times M$ matrix 
$$
\mathbb{E}(t)=[E(t_1)^{-1},\ldots,E(t_M)^{-1}]\quad \mbox{and}\quad 
\mathbb{K}(t)=[K(t_1)^{-1},\ldots,K(t_M)^{-1}]\in \mathbb{R}^{R^2\times M}
$$
obtained by concatenation of vectorized matrices $E(t_m)^{-1}$ and 
$K(t_m)^{-1}$, $m=1,\ldots,M$, respectively.
The idea is that 
\[
 E(t)^{-1}= \mbox{blockdiag}[P(t)^{-1},D(t)^{-1}] 
\]
defines an analytic matrix family on the spectral interval $t\in [0,a]$, and so is the family 
of core matrices $\{K^{-1}(t)\}$. This favorable property allows  
the model reduction via low rank approximation of the matrix families $\mathbb{E}(t)$ and 
$\mathbb{K}(t)$, $t\in \tau$. Suppose that the representations 
\[
 K(t)^{-1}=\sum\limits_{k=1}^{R_K} c_k(t) K_k
\]
and 
\[
 E(t)^{-1}= \mbox{blockdiag}[P(t)^{-1},D(t)^{-1}]= \sum\limits_{m=1}^{R_E} p_m(t) E_m
\]
are precomputed (this is an   additional low-rank approximation 
procedure which separates the parameter $t$),  
where $E_m=\mbox{blockdiag}[P_m,D_m]\in \mathbb{R}^{n\times n}$ and $K_k\in \mathbb{R}^{R\times R}$ 
do not depend 
on $t$, and $E_m$ inherits the block-diagonal structure that $E(t)^{-1}$ obeys.

We take into account that $Q$ does not depend on $t$, and 
plug the above decompositions in the main trace-term to obtain
\[
 \mbox{Trace}[E^{-1}Q K^{-1} Q^T  E^{-1}]= \mbox{Trace}
 \left[\sum\limits_{m=1}^{R_E} p_m(t) E_m \, Q \, 
 (\sum\limits_{k=1}^{R_K} c_k(t) K_k) \, Q 
 \sum\limits_{m'=1}^{R_E} p_{m'}(t) E_{m'}\right].
\]
Now it follows that
\[
 \mbox{Trace}[E^{-1}Q K^{-1} Q^T  E^{-1}]= 
 \sum\limits_{m=1}^{R_E} p_m(t) \sum\limits_{k=1}^{R_K} c_k(t)\sum\limits_{m'=1}^{R_E} p_{m'}(t)
  \mbox{Trace}[ E_m Q K_k Q E_{m'}],
 \]
where $K_k \in\mathbb{R}^{R\times R}$ is a small matrix, $Q \in\mathbb{R}^{n\times R}$,
$E_m =\mbox{blockdiag}[P_m,D_m]$ with diagonal $D_m$ and the full $n_P \times n_P$ matrix $P_m$,
such that $n_P=O({n}^\alpha)$. 

With these prerequisites, we pre-compute a set of  "time-independent" traces 
\begin{equation} 
\label{eqn:multi_trace}
T_{mkm'}= \mbox{Trace}[ E_m Q K_k Q E_{m'}], \quad m,m'=1,\ldots, R_E, \;\; k=1,\ldots, R_K,
\end{equation} 
and store the $R_E^2 R_K$ numbers $T_{mkm'}$ to obtain the cheap representation of the trace
in terms of only a scalar sum, 
\[
 \mbox{Trace}[E^{-1}Q K^{-1} Q^T  E^{-1}](t)=
 \sum\limits_{m=1}^{R_E}\sum\limits_{k=1}^{R_K}\sum\limits_{m'=1}^{R_E} p_{m}(t) c_k(t) p_{m'}(t) T_{mkm'}.
\]
The cost of precomputing each
trace-value $T_{mkm'}$ is estimated by $O(n^{3\alpha} R^2)$ as proven by Theorem \ref{thm:Trace_cost},
while the number of coefficients to be stored is about $O(R_E^2 R_K)$ and it is expected 
to be small or moderate.
With these data at hand, the evaluation of the required trace for the particular $t_\nu\in \tau$  
takes $O(R_E^2 R_K)$ scalar operations independently on $n$.

 Notice that the computations in (\ref{eqn:multi_trace}) are intrinsically parallel,
 which can be exploited on  modern computing hardware using multi-threading or distributed computing.

\section{QTT approximation of DOS}
\label{ssec:QTT_DOS}

In what follows, we discuss the QTT approximation of the DOS.
We also describe a tensor based   heuristic QTT approximation of the DOS 
by using only an incomplete set of sampling points, i.e., QTT representation by adaptive 
cross approximation (ACA) \cite{OselTyrt:2010,saOsel_cross:2011}.  
Furthermore, we derive the upper bound on the QTT ranks of the DOS by the Gaussians broadening.

\subsection{Quantized-TT approximation of function related vectors}
\label{ssec:QTT_Approx}

 In the case of large vector size $N$, the number of representation parameters for 
the corresponding high-order QTT tensor can be reduced to the logarithmic 
 scaling  $\mathcal{O}(\log N)$, which allows the QTT tensor interpolation  
 of the target $N$-vector by using only $\mathcal{O}(\log N)\ll N$ entries, 
 which are chosen adaptively  by the heuristic ACA algorithm \cite{OselTyrt:2010,saOsel_cross:2011}. 
 The accuracy of this kind of ``approximate interpolation'' is controlled 
 by the $\varepsilon$-truncation of the QTT rank parameters.  
In the present paper, we apply this approximation  technique  to long $N$-vectors 
 representing the DOS sampled over the fine representation grid $\Omega_h$.

The QTT-type approximation of an $N$-vector with $N=q^{d'}$, $d'\in \mathbb{N}$, $q=2,3,...$,
is defined as the tensor decomposition (approximation) in the TT or canonical format applied 
to a tensor obtained by the folding (reshaping) of the initial vector to a $d'$-dimensional 
$q\times \cdots \times q$ data array. The latter is thought of as an element of the 
multi-dimensional quantized tensor space $\mathbb{Q}_{{q},d'}= \bigotimes_{j=1}^{d'}\mathbb{K}^{q}$,
  $\mathbb{K}\in \{\mathbb{R},\mathbb{C}\}$, and $d'$ is the auxiliary dimension
 (virtual, in contrary to the real space dimension $d$)  parameter
that measures the depth of the quantization transform. A vector 
$
{ \bf x}=[x_i]_{i\in I}\in \mathbb{R}^N,
$
is reshaped to its multi-dimensional quantized image in $\mathbb{Q}_{q,d'}$ by $q$-adic folding, 
\[
\mathcal{F}_{q,d'}:{ {\bf x} \to \bf{X}}
=[x({\bf j})]\in \mathbb{Q}_{q,d'}, \quad {\bf j}=\{j_{1},\ldots,j_{d'}\}, 
\]
with  $j_{\nu}\in \{1,\ldots, q\}$ for $\nu=1,\ldots,d'$.
Here, for fixed $i$, we have $x({\bf j}):= x_i$, and $j_\nu=j_\nu(i)$ is defined via $q$-coding,
$
j_\nu - 1= C_{-1+\nu}, 
$
such that the coefficients $C_{-1+\nu} $ are found from the
$q$-adic representation  of $i-1$ (binary coding for $q=2$),
\[
i-1 =  C_{0} +C_{1} q^{1} + \cdots + C_{d'-1} q^{d'-1}\equiv
\sum\limits_{\nu=1}^{d'} (j_{\nu}-1) q^{\nu-1}.
\]
Assuming that for the rank-${\bf r}$ TT approximation of the quantized image ${ \bf{X}}$
there  holds $r_k \leq r$, $k=1,\ldots ,d'$, 
the complexity of  this  representation for the tensor ${ \bf{X}}$  
reduces to the logarithmic scale
$$
q r^2 \log_q N \ll N. 
$$ 
The computational gain of the QTT approximation is justified by the 
perfect rank decomposition proven in \cite{KhQuant:09} for 
a wide class of function-related tensors obtained by sampling the corresponding functions
over a uniform or properly refined grid. This class of functions includes
complex exponentials, trigonometric functions, polynomials 
and Chebyshev polynomials, as well as wavelet basis functions.
We refer to \cite{DKhOs-parabolic1-2012,osel-constr-2013,VeBoKh:Ewald:14,khor-survey-2014} 
for further results on QTT approximation of functional vectors and various applications.

In estimating the numerical complexity we use the average QTT rank further denoted by
$r_{qtt}$   calculated as follows,
\begin{equation}\label{eqn:av_rqtt}
 r_{qtt}= \sqrt{\frac{1}{d-1}\sum\limits^{d-1}_{k=1} r_k^2},
\end{equation}
where the QTT ranks $r_k$ are the TT ranks of the quantized image ${\bf X}$ of a vector.

As an example we present the basic results on the rank-$1$ (resp. rank-$2$)
QTT representation (with $q=2$) of the exponential (resp. trigonometric) vectors \cite{KhQuant:09}.
For given $N=2^{d'}$, and $z  \in \mathbb{C}$, the exponential $N$-vector, 
$
{ \bf z} :=\{z_n = z^{n-1}\}_{n=1}^N,
$
can be reshaped by the dyadic folding to the rank-$1$, $2^{\otimes d'}$-tensor,
\begin{equation}\label{eq exp-vect q}
\mathcal{F}_{2,d'}:  { \bf z} \mapsto  { \textbf{Z}}=
 \otimes_{p=1}^{d'} [1 \; z^{2^{p-1}}]^T  \in \mathbb{Q}_{{2},d'}.
\end{equation}
The number of representation parameters specifying the QTT image is reduced dramatically 
from $N$ to $2 \log_2 N $.

The trigonometric $N$-vector, 
$
 { \bf t}= \Im m({ \bf z}) :=\{t_n =  
 \sin(\omega (n-1))\}_{n=1}^N,\quad  \omega \in \mathbb{R},
$
can be reshaped by the successive dyadic folding 
\[
\mathcal{F}_{2,d'}:   { \bf t} \mapsto { \textbf{T}} \in \mathbb{Q}_{{2},d'},
\]
to the $2^{\otimes d'}$-tensor ${ \textbf{T}}$, which has both the canonical 
 and the QTT-rank equal to $2$,  in the complex and 
real arithmetics, respectively.  

The explicit rank-$2$ QTT-representation of the single $\sin$-vector in $\{0,1\}^{\otimes d'}$ 
(see \cite{dks-ttfft-2012,osel-constr-2013})
with $k_p=2^{p-1} i_p$, $i_p\in \{0,1\}$, reads
\begin{equation*}\label{eq sin-vect}
{\bf t} \mapsto {\textbf{T}}=\Im m ({\textbf{Z}})= 
[\sin\, \omega k_1 \cos\,\omega k_1] \otimes_{p=2}^{d'-1}
\left[
\begin{array}{c c}  \cos\,\omega k_p &-\sin\,\omega k_p \\  
\sin\,\omega k_p &\cos\,\omega k_p \end{array}
\right] \otimes \left[                                 
\begin{array}{c} \cos\,\omega k_{d'} \\ \sin\,\omega k_{d'} \end{array}
\right].
\end{equation*}
The number of representation parameters is $8 (d' -1)$. 
  A more detailed discussion of the QTT approximation 
for function related vectors can be found in \cite{KhQuant:09,khor-survey-2014}.

In cases when the exact low-rank QTT representation is not known, an
$\varepsilon$-approximation in the QTT format can be computed by using 
the standard TT multi-linear approximation tools \cite{Osel_TT:11}. 
As a first illustration, we consider the QTT approximation of the DOS for the 
1D finite difference Laplacian operator in $[0,\pi]$ with Dirichlet boundary 
conditions, $A=-\mbox{tridiag}\{1,-2,1\}\in \mathbb{R}^{n\times n}$,
discretized on the uniform grid of size $h=\pi/(n+1)$ with $n=2047$.
The corresponding eigenvalues are given by
\[
 \lambda_k=4\sin^2(\frac{\pi k}{2n}), \quad k=1,\ldots, n.
\]
Figure \ref{fig:DoS_QTTInterp_Lapl} represents the Lorentzian-DOS
and the corresponding approximation error for its QTT $\varepsilon$-interpolant with $r_{qtt}=5$, 
computed on the representation grid of size $N=2^{14}$.

\begin{figure}[htb]
\centering
\includegraphics[width=7.5cm]{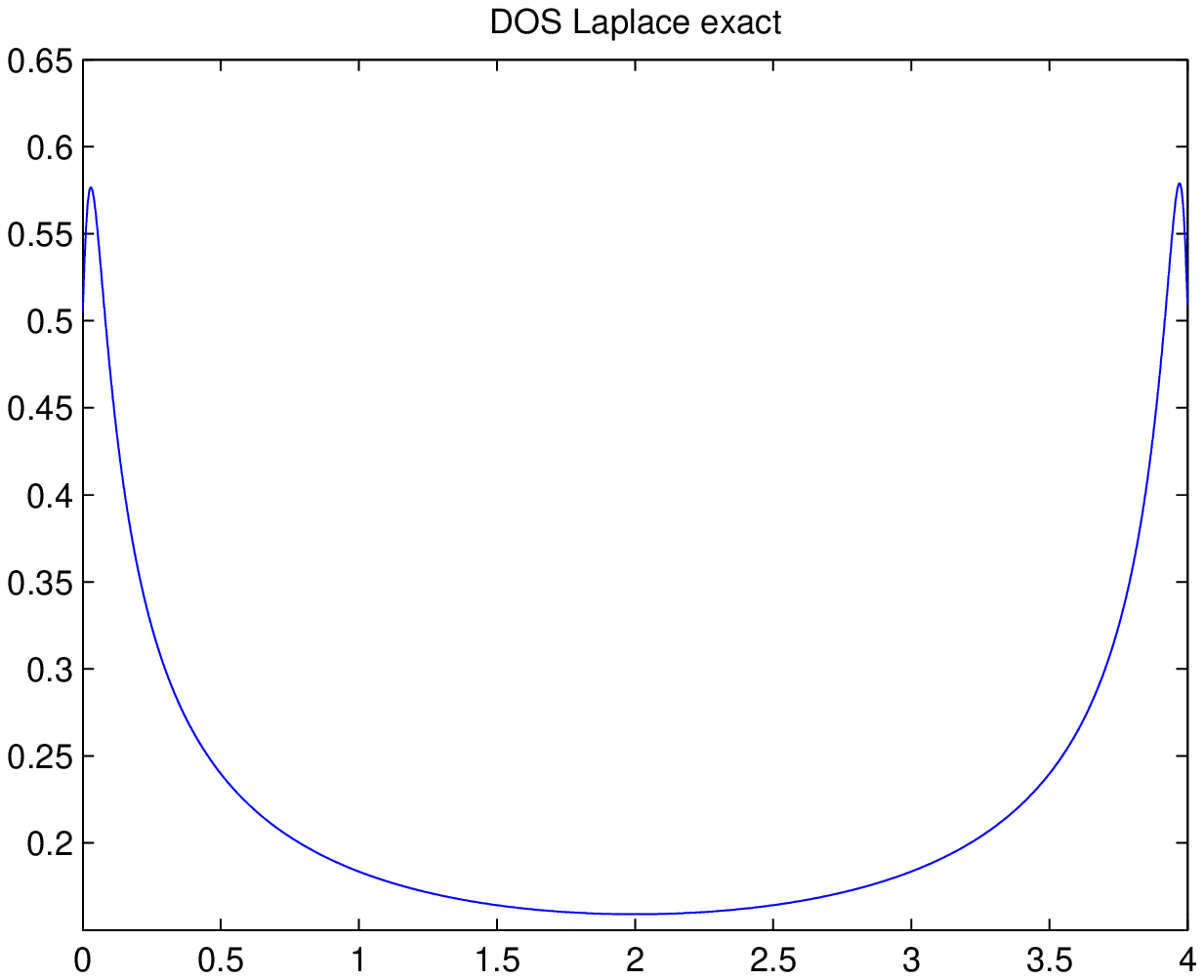}
\includegraphics[width=7.5cm]{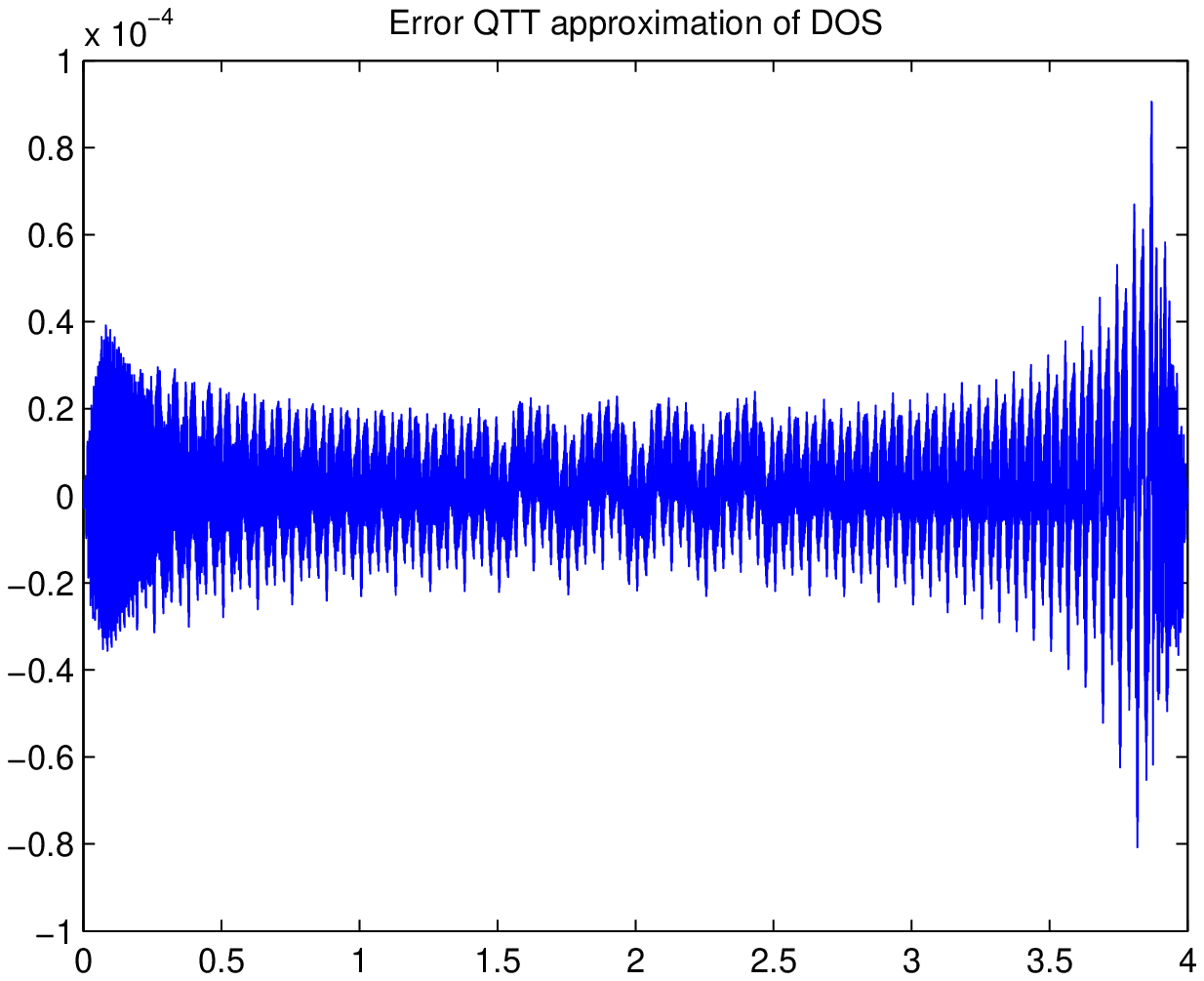}
\caption{\small  DOS for Laplacian (left),  and its QTT approximation with $r_{qtt}=5$ (right). }
\label{fig:DoS_QTTInterp_Lapl}
\end{figure}

In this paper we apply the QTT approximation method to the 
DOS regularized by Gaussians or Lorentzians and sampled on a fine representation grid of size $N=2^{d'}$.
The QTT approximant can be viewed as the rank structured $\varepsilon$-interpolant to the 
highly non-regular function $\phi_\eta$ regularizing the exact DOS.
In this case the application of traditional polynomial or trigonometric type 
interpolation is inefficient.

The QTT approach provides a good approximation to $\phi_\eta$ on the whole spectral interval
and requires only a moderate number of representation parameters $r_{qtt}^2 \log N \ll N$,
where the average QTT rank $r_{qtt}$, see (\ref{eqn:av_rqtt}) is a small rank parameter 
adaptively depending on the truncation error $\epsilon>0$.

\subsection{QTT approximation to DOS via Lorentzians: proof of concept} \label{ssesc: QTT_DOS}

In this section we demonstrate the efficiency of the QTT approximation  
applied to the DOS via both Gaussian and Lorentzian blurring. 
We verify by various numerical experiments that the low-rank QTT approximant 
resolves perfectly the exact  DOS. 
%

\begin{figure}[htb]
\centering
\includegraphics[width=7.0cm]{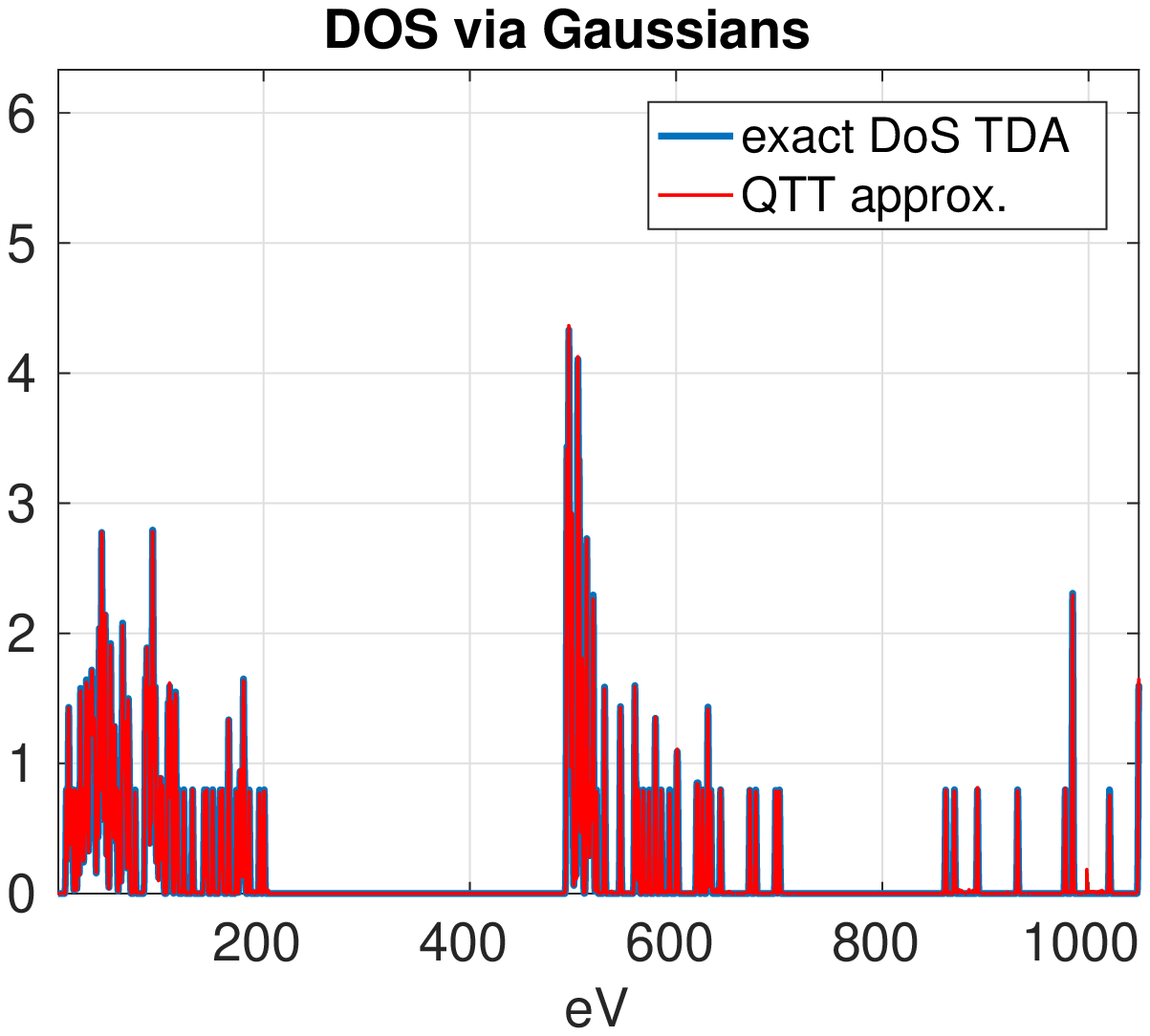} 
\includegraphics[width=7.0cm]{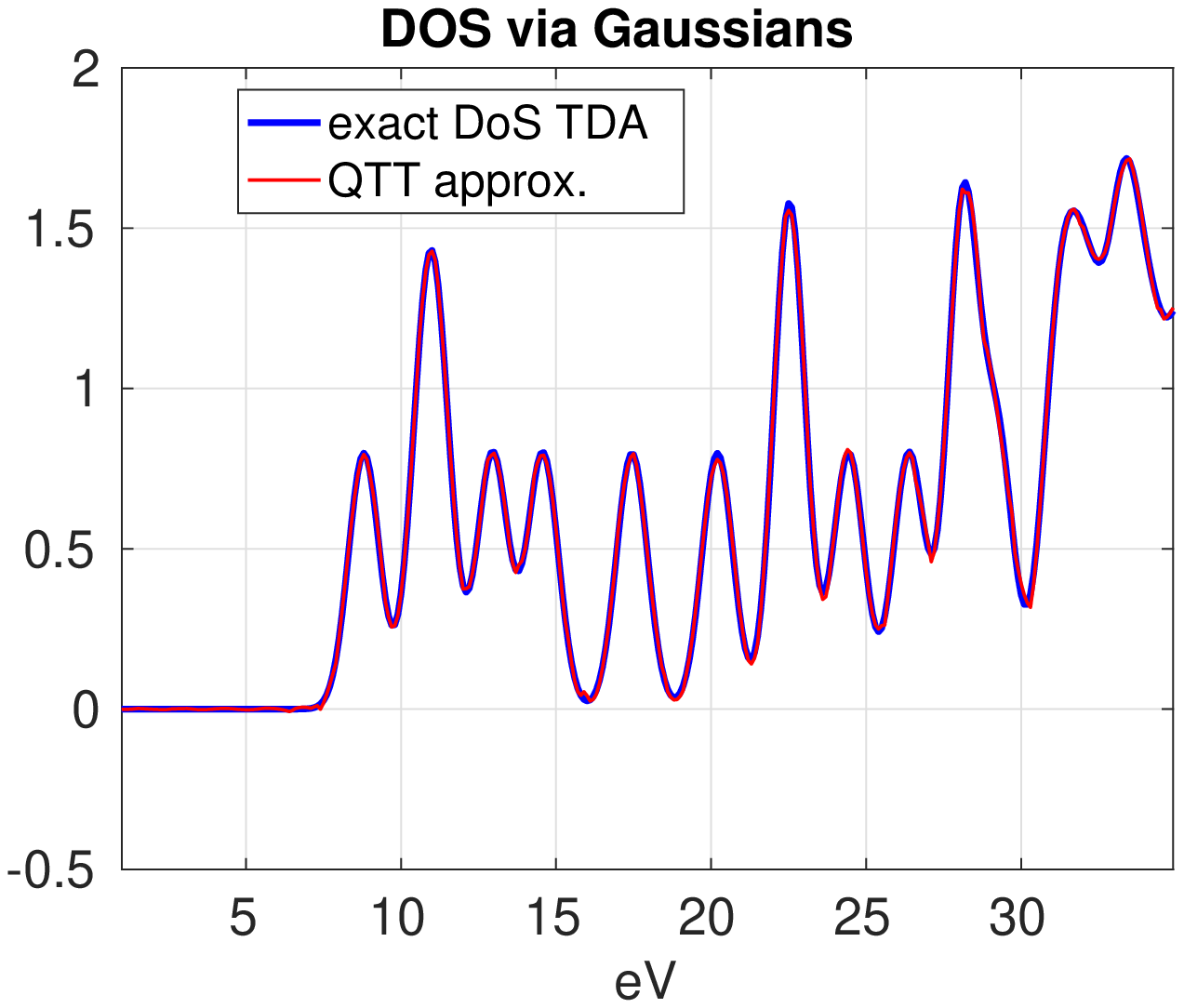}
\caption{\small DOS (in eV) for the H$_2$O molecule via Gaussians (left), and  
zoom on the  left most  part of the spectrum. Here $r_{QTT}=9.4$, $\eta=0.4$}
\label{fig:DoS_QTT_H2O}
\end{figure}

In the following numerical examples, we use a sampling vector defined on a 
grid of size $N \approx 2^{14}$. We set the QTT truncation error to $\epsilon_{QTT}=0.04$, 
if not explicitly indicated. 
For ease of interpretation, we set the pre-factor in (\ref{eqn:DOS}) to $1$. 
It is worth noting that the QTT-approximation scheme
is applied to the full TDA spectrum. Our results demonstrate
that it renders good resolution in the whole range of 
energies (in eV) including large "zero gaps".

Figure \ref{fig:DoS_QTT_H2O}, left, represents the TDA DOS (blue line)
for H$_2$O computed by Gaussian blurring with the parameter $\eta=0.4$,
and the corresponding rank-$9.4$ QTT tensor approximation (red line) to the 
discretized function $\phi_\eta(t)$. 
For this example, the number of eigenvalues is given by $n=N_{BSE}/2=180$. 
Figure \ref{fig:DoS_QTT_H2O}, right, provides a  zoom  of the corresponding 
DOS and its QTT approximant within the small energy interval $[0,40]$eV. 

\begin{figure}[htb]
\centering
\includegraphics[width=7.8cm]{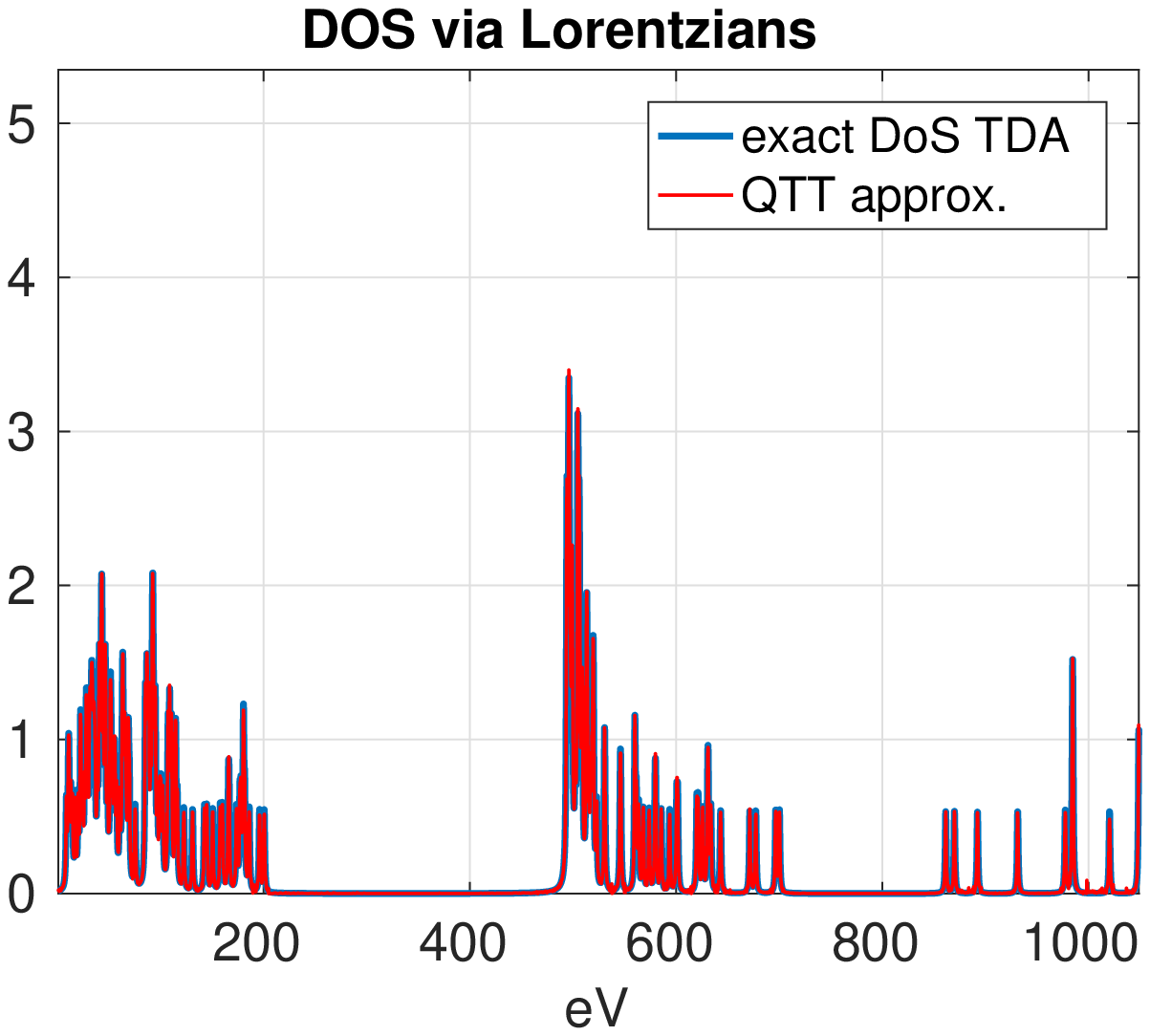}
\includegraphics[width=7.8cm]{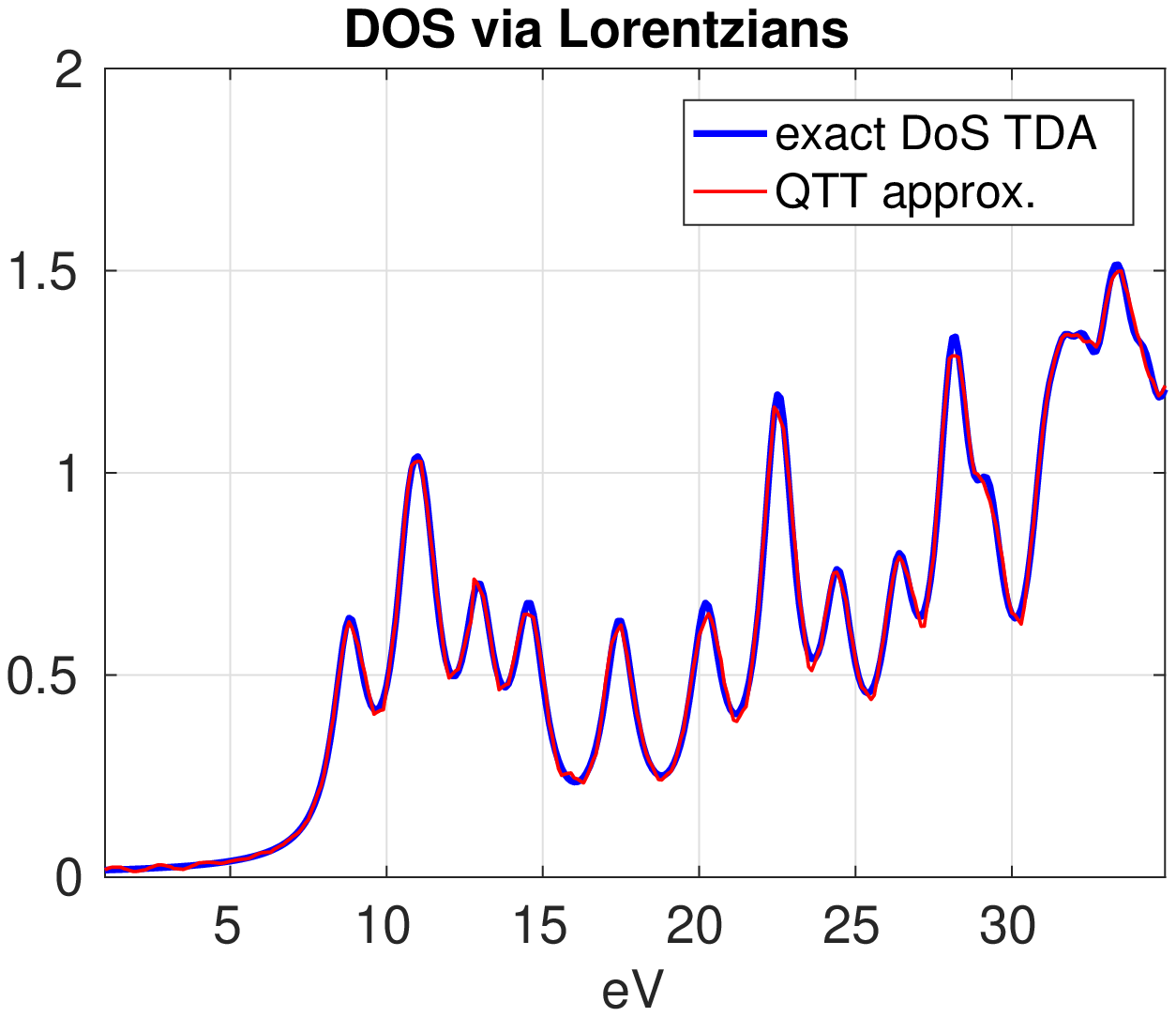} 
\caption{\small DOS for H$_2$O molecule  via Lorentzians (blue) 
and its QTT approximation (red) (left). Zoom on the left most part of the 
spectrum (right).  $\varepsilon$=0.04, $r_{QTT}=10.5$.}
\label{fig:DoS_QTT_H2O_Lor}
\end{figure}
Figure \ref{fig:DoS_QTT_H2O_Lor} demonstrates the resolution of the QTT approximation 
to the DOS via the Lorentzian blurring indicating similar QTT-ranks
as in the case of the Gaussians regularization. 

\begin{figure}[htb]
\centering
\includegraphics[width=7.8cm]{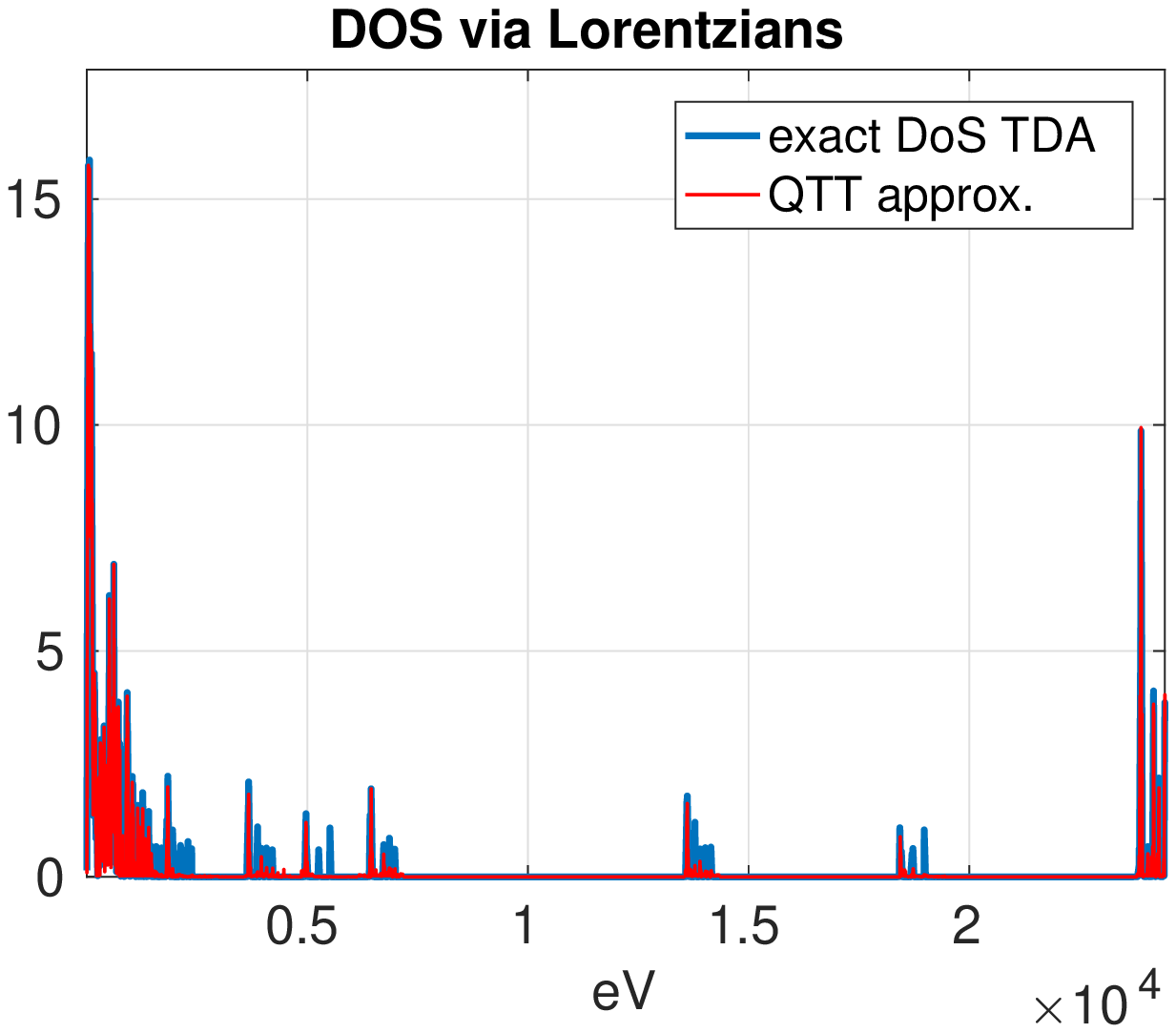}  
\includegraphics[width=7.8cm]{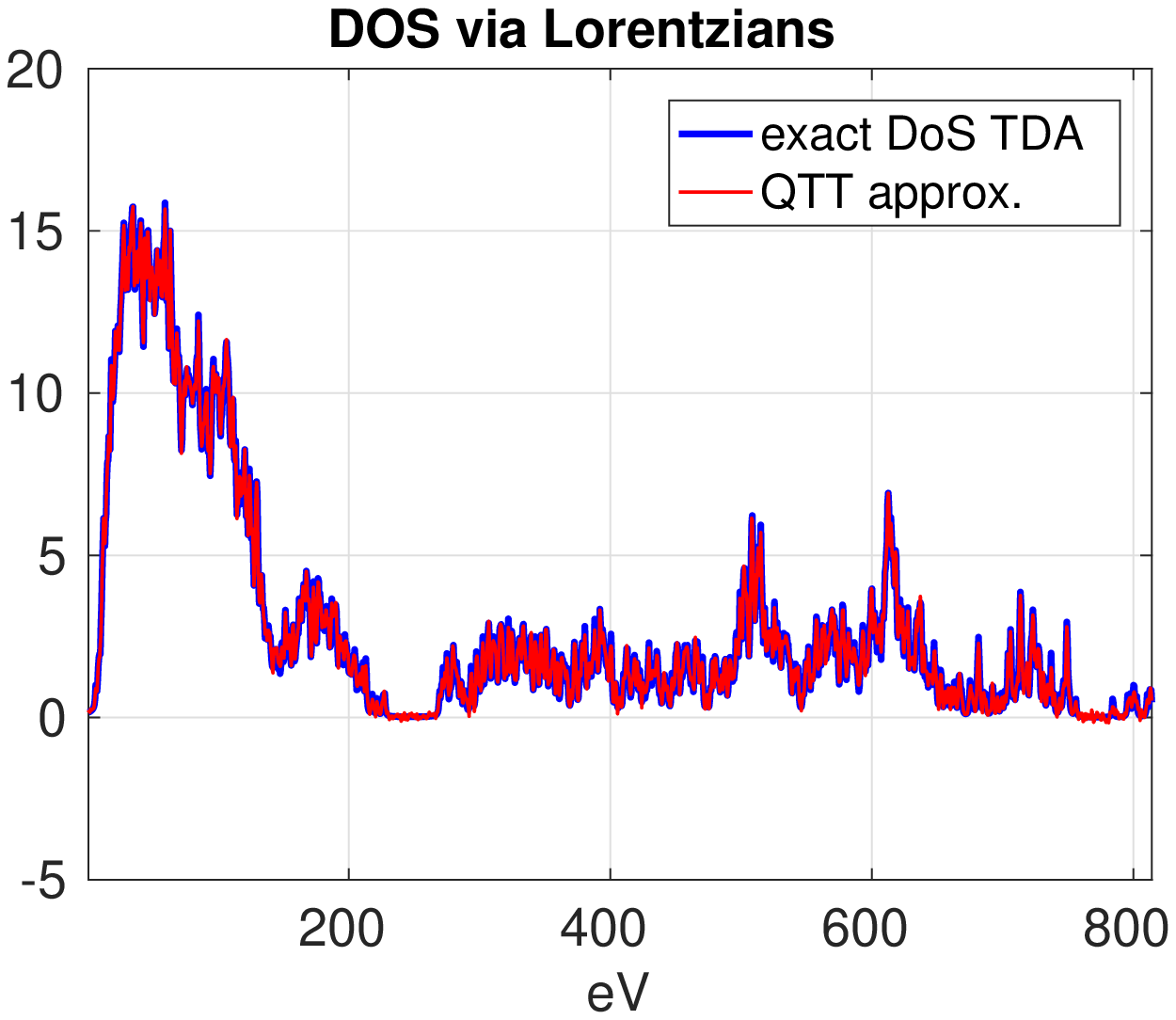} 
\caption{\small DOS for Glycine amino acid via Lorentzians (blue) 
and its QTT approximation (red), left; (left).  
Right: zoom of the first part of the spectrum.  $\varepsilon$=0.04, 
$r_{QTT}=16$.}
\label{fig:DoS_QTT_Gly}
\end{figure}
Figure \ref{fig:DoS_QTT_Gly} (Lorentzian blurring) represents similar data, but 
for the large Glycine amino acid with $n=N_{TDA}=3000$.
It is worth noting that the average QTT rank of $\phi_\eta(t)$ sampled 
on $N=2^{14}$ grid points is about $r_{QTT}=16$, ($\epsilon_{QTT}=0.04$)
though the number of eigenvalues $n$ in this case is about $20$ times larger 
than for the water molecule. 
This means that for a fixed $\eta$, the QTT-rank remains rather modest relative to the molecular size.
This observation confirms Theorem \ref{thm:QTT_R_Gaus_Broad} in Section  
\ref{ssec: QTT_ranks_DoS}.

A comparison of Figures \ref{fig:DoS_QTT_H2O} and \ref{fig:DoS_QTT_H2O_Lor}  
indicates that the Lorentzian based DOS blurring is slightly smoother than Gaussian blurring.
The moderate size of the QTT ranks in Figures \ref{fig:DoS_QTT_H2O_Lor} and \ref{fig:DoS_QTT_Gly}
clearly shows the potentials of the QTT $\varepsilon$-interpolation for modeling 
the DOS of large lattice type clusters.

We observe several gaps in the spectral densities, see Figure 
\ref{fig:DoS_QTT_H2O}, \ref{fig:DoS_QTT_H2O_Lor} and \ref{fig:DoS_QTT_Gly}
indicating that  polynomial, rational or trigonometric interpolation can be applied only to some 
energy sub-intervals, but not in the whole interval $[a,b]$.
Remarkably, the QTT approximant resolves well the DOS function in the whole energy 
interval including nearly zero values within the spectral gaps 
(hardly  possible for polynomial/rational based interpolation).

\subsection{Numerics for the QTT interpolation to the DOS function}
\label{ssec: QTT_cross_DoS}

In the previous section we demonstrated that the QTT tensor approximation provides
good resolution for the DOS function calculated for a number of molecules.
In what follows, we describe a tensor based   heuristic QTT approximation of the DOS 
by using only an incomplete set of sampling points, i.e., QTT representation by adaptive 
cross approximation (ACA) \cite{OselTyrt:2010,saOsel_cross:2011}.  
This allows us to recover the spectral density in controllable accuracy 
with  $M$ interpolation points, where $M$ asymptotically scales logarithmically 
in the grid size $N$. This heuristic approach can be viewed as a kind of 
  ``adaptive QTT $\varepsilon$-interpolation''.
In particular, we show by numerical experiments that the low-rank QTT adaptive cross interpolation 
 provides a good resolution of the target DOS 
with the number of functional calls that asymptotically scales logarithmically, $O(\log N)$,
in the size $N$ of the representation grid. 

In the case of large $N$, the QTT interpolant can be computed by the 
ACA tensor approximation procedure (see \cite{OselTyrt:2010,saOsel_cross:2011} for the detailed description)  
that, in general, does not require the full set of functional values over the $N$-grid.
In the case of large $N$ this beneficial feature allows to compute the QTT approximation 
by requiring   less   than $N$  computationally  expensive functional evaluations of $\phi_\eta(t)$.

The QTT interpolation via ACA tensor approximation serves to recover the representation 
parameters of the  QTT tensor approximant and normally requires about 
\begin{equation}\label{eqn:Cost_QTTint}
M=C_s r_{qtt}^2  \log_2 N
\end{equation}
samples of the target $N$-vector\footnote{In our application, 
this is the DOS functional $N$-vector corresponding to representations via matrix resolvents 
in (\ref{eqn:DOS_Loren_ImTr}) or (\ref{eqn:DOS_Loren_RTr}).} 
with a  small pre-factor $C_s$, usually satisfying $C_s\leq 10$,  that is independent of the fine 
interpolation grid size $N=2^{d'}$, see, for example, \cite{KhVe:16}. 
This cost estimate seems promising in the perspective of
extended or lattice type molecular systems,
requiring large spectral intervals and, as a result, a large interpolation
grid  of size $N$. Here the QTT rank parameter $r_{qtt}$ naturally depends on the 
required truncation threshold $\varepsilon>0$, characterizing the $L_2$-error between 
the exact DOS and its QTT interpolant.  The QTT tensor interpolation reduces the number of
functional calls, i.e., $M < N$, {\it if the QTT rank parameters (or threshold $\varepsilon>0$) 
are chosen to satisfy the condition}
\[
 M=C_s r_{qtt}^2  \log_2 N \leq N.
\]
The expression on the left-hand side provides  a  rather accurate estimate 
on the number of functional evaluations. 

To complete this discussion,
we present numerical tests for the low-rank QTT tensor interpolation  applied 
to the long vector discretizing the Lorentzian-DOS  on a fine representation grid of size 
$N=2^{d'}$. 
 \begin{figure}[htb]
\centering
 \includegraphics[width=15.0cm]{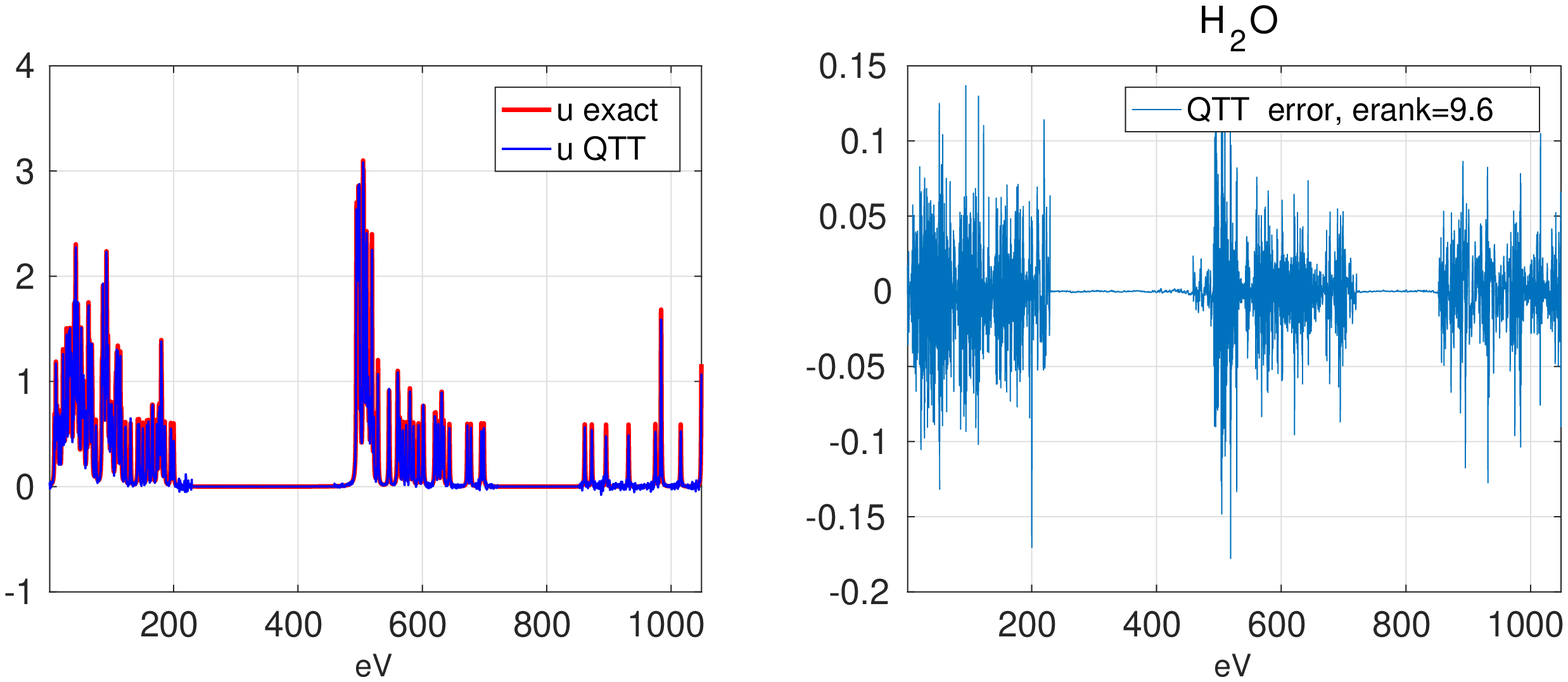}
 \includegraphics[width=15.0cm]{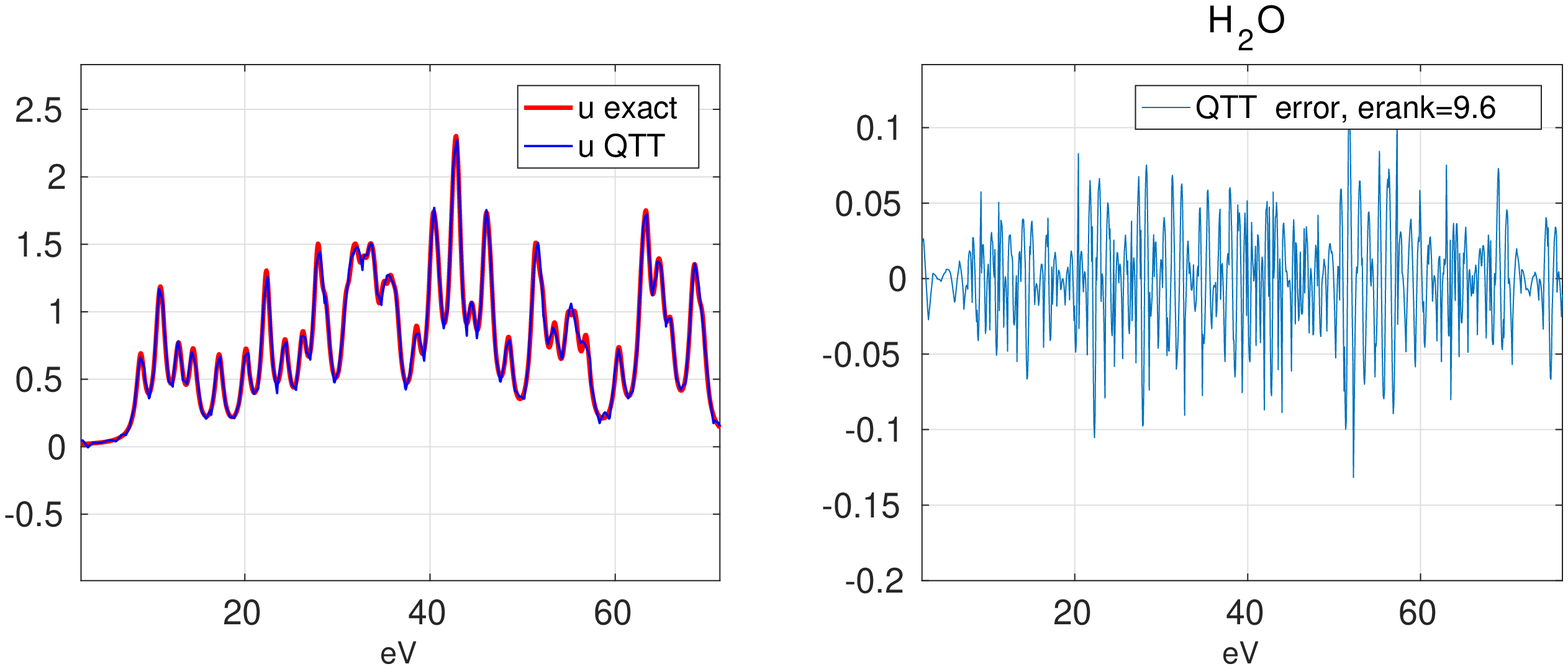}
\caption{\small QTT ACA interpolation of the DOS for H$_2$O (top) and zoom in to  
a small spectral interval (bottom).}
\label{fig:DoS_QTTcross_H2O}
\end{figure}

Figure \ref{fig:DoS_QTTcross_H2O} represents the results 
of the QTT interpolating approximation to the discretized DOS function (H$_2$O molecule).  
We use the QTT cross approximation algorithm based on 
\cite{KhQuant:09,OselTyrt:2010,saOsel_cross:2011,osel-constr-2013}
and implemented in  the MATLAB TT-toolbox (https://github.com/oseledets/TT-Toolbox). 
Here we set $\varepsilon=0.08$, $\eta=0.1$ and $N=2^{14}$, providing $r_{QTT}=9.8$.
The top two figures display the results on the whole spectral interval, while
the bottom figures show the zoom of the same data  in  the small spectral interval $[0,55]$eV.
\begin{figure}[htb]
\centering
\includegraphics[width=7.0cm]{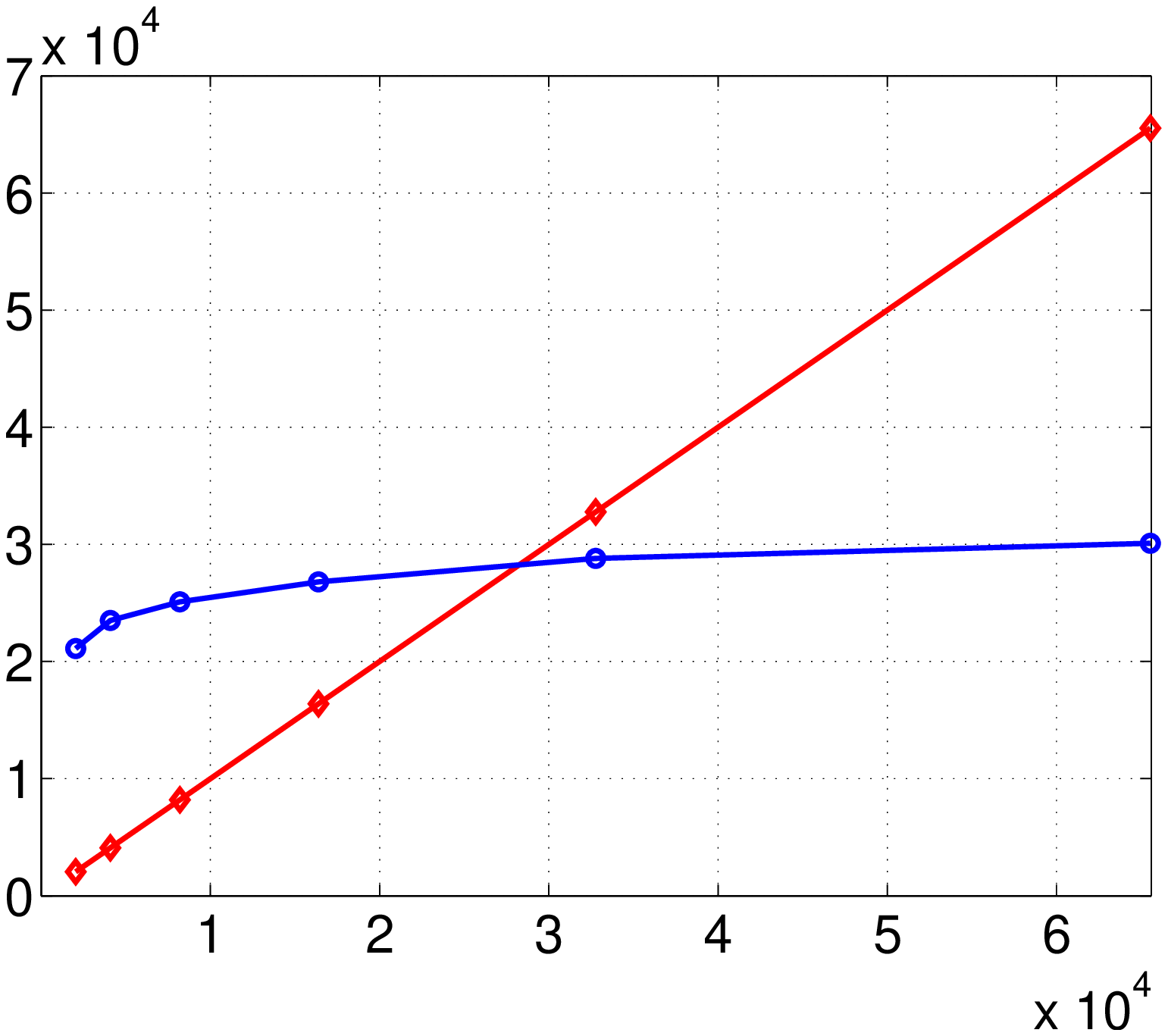}
\caption{\small DOS for H$_2$O via Lorentzians: the number of functional calls
for QTT cross approximation (blue) vs. the full grid size $N$.}
\label{fig:DoS_QTTInterp_vsN}
\end{figure}

Figure \ref{fig:DoS_QTTInterp_vsN} illustrates the logarithmic increase in the number of 
samples  required for the 
QTT interpolation of the DOS (for the H$_2$O molecule) represented on the grid of size $N=2^{d'}$,
where $d'=11,12,\ldots,16$, provided that the rank truncation threshold is chosen by 
$\epsilon=0.05$ and  the  regularization parameter is $\eta=0.2$. 
In this example, the effective pre-factor in (\ref{eqn:Cost_QTTint})
is estimated by $C_s\leq 10$.  This pre-factor characterizes the average number of samples 
required for the recovery of each of the $r_{qtt}^2  \log N$ representation  parameters  involved in
the QTT tensor ansatz.  

We observe that the QTT tensor interpolant recovers the exact DOS with  a  high precision.
The logarithmic asymptotic complexity scaling $O(\log N)$ (i.e. the number of functional calls 
required for the QTT tensor interpolation) vs. the grid size $N$ can be observed 
in Figure \ref{fig:DoS_QTTInterp_vsN} (blue line).

\subsection{Upper bounds on the QTT ranks of DOS }\label{ssec: QTT_ranks_DoS}

In this section we analyze the upper bounds on the QTT ranks of the discretized DOS
obtained by Gaussian  broadening.  Our numerical tests indicate that Lorentzian 
blurring leads to a similar QTT rank compared with Gaussians blurring when 
both are applied to the same grid and the same truncation threshold 
$\varepsilon>0$ is used in the QTT approximation. 
We consider the more general case 
of a symmetric interval, i.e.  $t, \lambda_j\in [-a,a]$.

Assume that the function 
$\phi_\eta(t)= \frac{1}{n} \sum\limits_{j=1}^{n} g_\eta(t -\lambda_j)$, $t\in [-a,a]$, 
in equation (\ref{eqn:DOS_gauss}) is discretized by sampling over the uniform $N$-grid 
 $\Omega_h$ with $N=2^d$, where  the generating Gaussian is given by
$g_\eta(t)=\frac{1}{\sqrt{2 \pi}\eta}\exp{\left(-\frac{t^2}{2 \eta^2}\right)}$.
Denote the corresponding $N$-vector by ${\bf g}={\bf g}_\eta$, and the resulting 
discretized density vector by 
$$
\phi_\eta(t) \mapsto {\bf p}={\bf p}_\eta=
 \frac{1}{n} \sum\limits_{j=1}^{n} {\bf g}_{\eta,j}  \in \mathbb{R}^N,
$$
where  the shifted Gaussian is assigned by the vector 
$g_\eta(t -\lambda_j)\mapsto {\bf g}_j={\bf g}_{\eta,j}$.

Without loss of generality, we
suppose that all eigenvalues are situated within the set of grid points, i.e.
$\lambda_j\in \Omega_h$. Otherwise, we can slightly relax their positions provided that 
the mesh size $h$ is small enough. This is not a  severe  restriction for the QTT 
approximation  of functional vectors since storage and complexity requests
depend only logarithmically on $N$.

\begin{theorem}\label{thm:QTT_R_Gaus_Broad}
Assume that the effective support of  the  shifted Gaussians 
$g_\eta(t -\lambda_j)$, $j=1,\dots,n$, is included
 in   the computational interval $[-a,a]$. Then the QTT $\varepsilon$-rank of the 
vector ${\bf p}_\eta$ is bounded by
\[
 rank_{QTT}({\bf p}_\eta)\leq C a \log^{3/2}(|\log \varepsilon|),
\]
where the constant $C=O(|\log \eta|) >0$ depends only logarithmically on the 
regularization parameter $\eta$.
\end{theorem}
\begin{proof} 
The main argument of the proof is similar to that in 
\cite{VeBoKh:Ewald:14,DKhOs-parabolic1-2012}:
the sum of discretized Gaussians, each represented in Fourier basis,
can be expanded with merely the same number of Fourier harmonics (uniform basis) 
as each individual Gaussian. 

Now we estimate the number of essential 
Fourier coefficients of the Gaussian vectors ${\bf g}_{\eta,j}$ with a fixed exponent 
parameter $\eta$,
\[
 m_0=O(a |\log \eta| \log^{3/2}(|\log \varepsilon|)),
\]
taking into account their exponential decay. Here $\varepsilon>0$ denotes the 
rank truncation threshold. Notice that $m_0$ depends logarithmically on $\eta$.
Since each Fourier harmonic has exact rank-$2$ QTT representation 
(see Section \ref{ssec:QTT_Approx}), we arrive at the claimed bound.
\end{proof}

Notice that the Fourier transform of the Lorentzian in (\ref{eqn:Delta_Lorentz})
is given by 
\[
 e^{-|k|\eta},
\]
thus a similar QTT rank bound can be derived for the case of Lorentzian blurred DOS. 

\begin{table}[hbp]
 \begin{center}
 \begin{tabular}
[c]{|c|c|c|c|c|c|c|c|}%
\hline
Molecule & H$_2$O & NH$_3$ & H$_2$O$_2$ & N$_2$H$_4$ &C$_2$H$_5$OH & C$_2$H$_5$ NO$_2$ & C$_3$H$_7$ NO$_2$ \\
 \hline
 $n=N_{ov}$ & $180$ & $215$ & $531$ & $657$ & $1430$   & $3000$ & $4488$   \\
 \hline
QTT ranks & $11$ & $11$ & $12$ & $11$ & $15$ & $16$ & $13$   \\
\hline
\end{tabular}
\end{center}
\caption{\small QTT ranks of Lorentzians-DOS for some molecules; $\varepsilon=0.04$, $\eta=0.4$, $N=16384$.}
\label{tab:QTT_ranks_Lor}
\end{table} 
Table \ref{tab:QTT_ranks_Lor} shows that the average QTT tensor rank 
remains almost independent of the molecular size, which confirms 
Theorem \ref{thm:QTT_R_Gaus_Broad}.
The weak dependence of the rank parameter on the molecular geometry can be observed.

\section{Towards calculation of the BSE absorption spectrum}\label{sec:BSE_case}

In this section we describe the generalization of
our approach to the case of the full BSE system.
Within the BSE framework, the optical absorption spectrum of a molecule is defined by
\begin{equation}
\epsilon(\omega) \equiv d_r \herm\delta(\omega I_{2n}-H)d_l
=\sum_{j=1}^{2n}
\frac{(d_r\herm (z_r)_j)((z_l)_j\herm d_l)}{(z_l)_j\herm(z_r)_j}
\delta(\omega-\lambda_j),
\label{eq:bseabs}
\end{equation}
where \[
d_r=\begin{bmatrix} d \\ -\conj d\end{bmatrix}
\quad \text{and} \quad
d_l=\begin{bmatrix} d \\ \conj d\end{bmatrix}
\]
are the right and left \emph{optical transition vectors}, respectively, and $d$ is a vector \
reshaped from a transition matrix $T$ of dimension $N_o \times (N_b - N_o)$. 
The $(i,a)$th element of $T$ is given by $\langle \psi_i | \vec{x} | \psi_a \rangle$, where $\vec{x}$ is a \
position operator in the direction of $x$ and $\psi_i$ and $\psi_a$ are 
a pair of occupied and unoccupied molecular \
orbitals~\cite{molgw}. 

Similar to the DOS, the function $\epsilon(\omega)$ is a sum
of Dirac-$\delta$ peaks centered at eigenvalues of the BSH. 
However, the height of each peak, which is often referred to 
as the oscillator strength, is determined by the projection of
the corresponding left and right eigenvectors of $H$ onto the optical 
transition vectors $d_l$ and $d_r$.

A smooth approximation of \eqref{eq:bseabs} can be obtained by 
replacing the Dirac-$\delta$ function with either a Gaussian 
or a Lorentzian with an appropriate broadening width.  If we choose
to smooth by a Lorentzian, we then need to compute
\begin{equation}
\epsilon(\omega) \approx \frac{1}{\pi}\mbox{Im}\biggl[ d_r\herm 
\left(\omega I_{2n}-H -i \eta I_{2n} \right)^{-1} d_l\biggr],
\label{eq:lorabs}
\end{equation}
where $\eta$ is related to the width of broadening. 

For a fixed frequency $\omega$, \eqref{eq:lorabs} can be evaluated by
solving a linear system of the form
\[
\left(\omega I_{2n}-H -i \eta I_{2n} \right) x = d_l.
\]
The block sparse and low-rank structure of $H$ can be used to reduce
the cost for solving such a linear system.

The detailed numerical analysis of this scheme for the BSE system  will be
a  topic of a forthcoming paper.

\section{Conclusions}

The new approach to approximating the DOS of the TDA of a BSE Hamiltonian
is based on two main techniques.
First, we take advantage of the low rank structure of the TDA and evaluate the trace of the resolvent 
directly instead of using stochastic sampling techniques.
The presented economical algorithm provides an efficient way 
to calculate the DOS regularized by Lorentzians. The cost of
the computation scales linearly with respect to the matrix size.
Second, a QTT based tensor interpolation scheme is used to
approximate the DOS discretized on large representation grids.
This  approximation  scheme allows us to estimate the DOS
with $M$ function evaluations, where $M$ scales logarithmically
with respect to the grid size on which the DOS is evaluated.  
The approach can be applied to a wide class of rank-structured symmetric spectral problems.

In Theorems \ref{thm:Trace_cost} and \ref{thm:Trace_cost_real}, we prove linear scaling of the structured 
trace calculation algorithm in the matrix size.
This result is confirmed by numerical experiments performed 
to compute the DOS of BSH associated with some molecular systems as shown in 
Figure \ref{fig:DoS_trace_Times}.

We justify the low rank QTT approximation of the DOS in the case of Gaussian regularization, 
see Theorem \ref{thm:QTT_R_Gaus_Broad}.
The efficiency of low-rank QTT approximation to DOS is illustrated numerically on the example
of discrete Laplacian as well as for the BSE spectral problem for several moderate size molecules.
Numerical tests demonstrate the logarithmic complexity 
of the QTT cross approximation scheme  in the grid size, applied to the discretized DOS
as depicted in Figure \ref{fig:DoS_QTTInterp_vsN}.

It is worth noting that our approach serves to recover DOS on the whole 
spectral interval which is demonstrated in a number of numerical tests.
However, the algorithms are applicable to any fixed subinterval of interest in the whole
spectrum, which will correspondingly reduce the QTT tensor ranks and 
the overall computational time.

The presented methods introduce a new efficient tool for 
numerical approximation to the DOS for large matrices arising in various applications
in condensed matter physics, computational quantum chemistry as well as 
in large-scale problems of numerical linear algebra.

\begin{footnotesize}

\bibliographystyle{abbrv}
\bibliography{BSE_AbsSpectr}
\end{footnotesize}

\end{document}